\documentclass[11pt]{amsart}
\usepackage{mathptmx,amsthm,amsmath,graphicx,tikz,comment}
\usepackage{array}
\usetikzlibrary{calc}
\usetikzlibrary{decorations.markings}
\usetikzlibrary{matrix}
\usetikzlibrary{arrows}

\newtheorem{theorem}{Theorem}[section]
\newtheorem{proposition}[theorem]{Proposition} 
\newtheorem{lemma}[theorem]{Lemma}
\newtheorem{corollary}[theorem]{Corollary}
\newcommand{\vlambda}{{\vec{\lambda}}}
\newcommand{\vgamma}{{\vec{\gamma}}}
\newcommand{\vmu}{{\vec{\mu}}}
\newcommand{\vnu}{{\vec{\nu}}}
\newcommand{\Gr}{\mathrm{Gr}}
\newcommand{\SL}[1][n]{\mathrm{SL_#1}}
\newcommand{\Hom}{\mathrm{Hom}}
\newcommand{\C}{\mathbb{C}}
\newcommand{\D}{\mathbb{D}}
\newcommand{\IC}{\mathrm{\mathbf{IC}}}
\newcommand{\PGL}{\mathrm{PGL_n}}
\renewcommand{\d}{\mathrm{d}}
\newcommand{\cK}{\mathcal{K}}
\newcommand{\cO}{\mathcal{O}}
\newcommand{\rep}{\mathrm{\mathbf{rep}}}
\renewcommand{\top}{\mathrm{top}}
\newcommand{\braket}[1]{{\langle #1 \rangle}}
\newcommand{\perv}{\mathrm{\mathbf{perv}}}

\begin{document}

\title{Generating Basis Webs for $\SL$}
\author{Bruce Fontaine}
\address{University of Toronto}

\begin{abstract} 
Given a simple algebraic group $G$, a web is a directed trivalent graph with edges labelled by dominant minuscule weights. There is a natural surjection of webs onto the invariant space of tensor products of minuscule representations. Following the work of Westbury, we produce a set of webs for $\SL$ which form a basis for the invariant space via the geometric Satake correspondence. In fact, there is an upper unitriangular change of basis to the Satake basis. This set of webs agrees with previous work in the cases $n=2,3$ and generalizes the work of Westbury in the case $n\geq 4$.
\end{abstract}

\maketitle

\section{Introduction}

Our goal of this paper is to extend \cite{Kuperberg:spiders} and \cite{Westbury:general} and produce a basis for the invariant space $(V_{1}\otimes\cdots\otimes V_{k})^{\SL}=\Hom_{\SL}(\C,V_{1}\otimes\cdots\otimes V_{k})$ with the $V_i$ a minuscule $\SL$ representations. This basis will have the special property that there is an upper unitriangular change of basis to the Satake basis. In \cite{fkk}, given a sequence $\vlambda$ of dominant minuscule weights for a simple, simply connected algebraic group $G$, vectors in $\Hom_G(\C,V_{\vlambda})$ where $V_\vlambda=V_{\lambda_1}\otimes V_{\lambda_2}\otimes\cdots\otimes V_{\lambda_k}$ can be obtained from a simple combinatorial object called a web. A {\em web} is a directed, trivalent graph whose edges are labelled by the dominant minuscule weights of $G$. Each vertex in a web can be thought of as a vector in $\Hom_G(\C,V_{\vnu})$ where $\vnu$ is the sequence of edge labels at the vertex. A web is then a way of composing the vectors corresponding to the vertices to obtain a new vector in some space $\Hom_G(\C,V_{\vlambda})$. Here $\vlambda$ is determined by the web and is called its {\em boundary}. The vector in $\Hom_G(\C,V_{\vlambda})$ corresponding to a web $w$ is called the {\em web vector} and will be denoted $\Psi(w)$.

In the case $G=\SL$ all fundamental weights are minuscule. Let $\omega_i$ be the $i$-th fundamental weight of $\SL$, i.e. the highest weight of the irreducible representation $\bigwedge^i\C^n$. In \cite{Morrison:diagram} it is shown that for each sequence $\vlambda$ of dominant minuscule weights, the image of all webs with boundary $\vlambda$ under $\Psi$ spans $\Hom_{\SL}(\C,V_\vlambda)$. In the case $n=2$, the set of crossless planar matchings are webs and form a basis for the invariant spaces under $\Psi$. In \cite{Kuperberg:spiders} Kuperberg shows that in the case $n=3$, the set of elliptic webs forms a basis for $\Hom_{\SL[3]}(\C,V_\vlambda)$ under $\Psi$. More recently in \cite{Westbury:general} Westbury provided a basis for general $n$ when $\lambda_i\in\{\omega_1,\omega_{n-1}\}.$

Let $\cO = \C[[t]]$ and $\cK = \C((t))$. The {\em affine Grassmannian} of $\PGL$ is $$\Gr = \Gr(\PGL) = \PGL(\cK)/\PGL(\cO)$$ where $\PGL$ is the Langlands dual of $\SL$. The left $\PGL(\cO)$ orbits on $\Gr$ are in one-to-one correspondence with dominant weights of $\SL$. Each orbit is generated by an element denoted $t^\lambda$. The left $\PGL(\cK)$ diagonal orbits on $\Gr\times\Gr$ are in one-to-one correspondence with the left $\PGL(\cO)$ orbits on $\Gr$. Thus these diagonal orbits are also labelled by dominant weights. Let $\d(L,L')$ be the label of the orbit in which the pair $(L,L')$ resides. Then $\d$ is a weight valued metric in the sense of \cite{KLM:generalized} and with it we can define the {\em twisted product} $$\overline{\Gr_\vlambda}=\{(L_0,\cdots,L_n)\in\Gr^{n+1}|L_0=t^0,\d(L_{i-1},L_i)\leq\vlambda\}$$ which has open stratum $$\Gr_\vlambda=\{(L_0,\cdots,L_n)\in\Gr^{n+1}|L_0=t^0,\d(L_{i-1},L_i)=\vlambda\}.$$ Let $m_\vlambda:\overline{\Gr_\vlambda}\rightarrow\Gr$ be the projection onto the last factor. $m_\vlambda$ is also known as the {\em multiplication morphism} and its fibres $F(\vlambda)=m^{-1}_\vlambda(t^0)$ are called {\em Stake fibres}.

We have the following result of the Geometric Satake correspondence:

\begin{theorem} Let $\vlambda$ be a sequence of weights of $\SL$, then $\Hom_{\SL}(\C,V_\vlambda)$ is canonically isomorphic to $H_\top(F(\vlambda),\C).$ Thus each top-dimensional component $Z \subseteq F(\vlambda)$ yields a vector $[Z] \in \Hom_{\SL}(\C,V_\vlambda)$. These vectors form a basis, the {\em Satake basis}.
\end{theorem}

Given a sequence of dominant minuscule weights $\vlambda$, a {\em minuscule Littelmann path} of type $\vlambda$ is a sequence $\vmu$ of dominant weights starting and ending with $0$ such that $\mu_i-\mu_{i-1}\in W\lambda_i$ for all $i$. Let $P_\vlambda$ be the set of all such paths. Then, when $\vlambda$ is minuscule the components of $F(\vlambda)$ have a natural labelling by the elements of $P_\vlambda$ and $\Gr_\vlambda=\overline{\Gr_\vlambda}$. The set $P_\vlambda$ has a partial order: for all $\vmu,\vmu'\in P_\vlambda$ we set $\vmu'\leq\vmu$ if and only if $\mu_i'\leq\mu_i$ for all $i$. For each $\vmu\in P_\vlambda$ we will construct a web $w_\vmu$ with boundary $\vlambda$. The web $w_\vmu$ will satisfy certain coherence conditions that will allow the calculation of the expansion of $\Psi(w_\vmu)$ in terms of the Satake basis:
\begin{theorem}\label{thm:main} The image of $\{w_\vmu|\vmu\in P_\vlambda\}$ under $\Psi$ is a basis of $\Hom_{\SL}(\C,V_\vlambda)$ and the change of basis to the Satake basis is upper unitriangular with respect to the partial ordering on $P_\vlambda$.
\end{theorem}

In \cite{fkk} for any web $w$ with boundary $\vlambda$, we produced an associated {\em web variety} $Q(D(w))$ with a projection $\pi$ to $F(\vlambda)$ and a class $c(w)\in H_\top(Q(D(w)))$ such that $\pi_*(c(w))=\Psi(w)$. By understanding $c(w)$, we can show that its image under $\pi_*$ expands in terms of the Satake basis as described above. In fact, the coherence conditions on a web will allow us to associte to each coherent web with boundary $\vlambda$ a path $\vmu\in P_\vlambda$. The invariant vector $\Psi(w)$ coming from a coherent web $w$ expands in the Satake basis as the component of $F(\vlambda)$ labelled $\vmu$ plus lower order terms. This fact will follow by adapting the $\SL[3]$ results in \cite{fkk} to $\SL$. Thus if one selects $w_\vmu$ to be a coherent web with path $\vmu\in P_\vlambda$ then the theorem would be proven. The construction of these diagrams will be given by a generalization of the triangular diagram construction of Westbury from \cite{Westbury:general} to the case of arbitrary dominant minuscule weights.

\section{A geometric interpretation of webs}

\subsection{Webs and Dual Diskoids}
Formally, a {\em web} for $\SL$ is a morphism in the free spider defined in \cite{fkk}. It is an oriented trivalent graph in the disk with a marked boundary point with edges labelled by minuscule representation of $\SL$. The orientation of an edge can be switched at the cost of dualizing the labelling representation. We then require that for each vertex in the graph when the edges are oriented inwards the labelling representations $V_\lambda$, $V_\mu$ and $V_\nu$ satisfy $$\Hom_{\SL}(\C,V_\lambda\otimes V_\mu\otimes V_\nu) \ne 0.$$ The condition that the weights $\lambda$, $\mu$ and $\nu$ are minuscule forces this vector space to be at most one-dimensional. The {\em boundary} of a web is given by orienting all boundary edges outwards and reading the labels of the edges clockwise from the marked point. Figure~\ref{sl4web} shows an example web for $\SL[4]$.
 
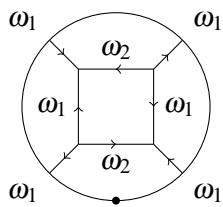
\begin{figure*}
\begin{tikzpicture}[ar/.style={postaction={decorate,decoration={markings,mark=at position .5 with {\arrow{>}}}}},scale=0.5]
\draw[ar] (1,1) -- (1,-1) node[midway,right] {$\omega_1$};
\draw[ar] (1,1) -- (-1,1) node[midway,above] {$\omega_2$};
\draw[ar] (1,1) -- (45:2.5) node[above right] {$\omega_1$};
\draw[ar] (-1,-1) -- (1,-1) node[midway,below] {$\omega_2$};
\draw[ar] (-1,-1) -- (-1,1) node[midway,left] {$\omega_1$};
\draw[ar] (-1,-1) -- (225:2.5) node[below left] {$\omega_1$};
\draw[ar] (135:2.5) node[above left] {$\omega_1$} -- (-1,1);
\draw[ar] (315:2.5) node[below right] {$\omega_1$} -- (1,-1);
\draw (0,0) circle(2.5);
\fill (270:2.5) circle(3pt);
\end{tikzpicture}
\caption{A web for $\SL[4]$. The boundary is $\omega_1,\omega_3,\omega_1,\omega_3$.}
\label{sl4web}
\end{figure*}

Given a web $w$, define its {\em dual diskoid} $D(w)$ as the dual graph with edges oriented by the requirement that the pairing of each edge with its dual gives a positive orientation. The edges of $D(w)$ take their labelling from the edges of $w$ and the marked boundary point on the disc gives a marked vertex. Figure~\ref{sl4webdual} shows the dual diskoid for Figure~\ref{sl4web}.

\begin{figure*}
\begin{tikzpicture}[ar/.style={postaction={decorate,decoration={markings,mark=at position .5 with {\arrow{<}}}}},scale=0.75]
\draw[ar] (0,2) -- (0,0) node[midway,left] {${\scriptstyle\omega_2}$};
\draw[ar] (0,-2) -- (0,0) node[midway,right] {${\scriptstyle\omega_2}$};
\draw[ar] (0,0) -- (2,0) node[midway,above] {${\scriptstyle\omega_1}$};
\draw[ar] (0,0) -- (-2,0) node[midway,below] {${\scriptstyle\omega_1}$};
\draw[ar] (2,0) -- (0,2) node[midway,above right] {$\omega_1$};
\draw[ar] (-2,0) -- (0,2) node[midway,above left] {$\omega_1$};
\draw[ar] (2,0) -- (0,-2) node[midway,below right] {$\omega_1$};
\draw[ar] (-2,0) -- (0,-2) node[midway,below left] {$\omega_1$};
\fill (0,-2) circle(3pt);
\end{tikzpicture}
\caption{The dual diskoid of Figure~\ref{sl4web}.}
\label{sl4webdual}
\end{figure*}
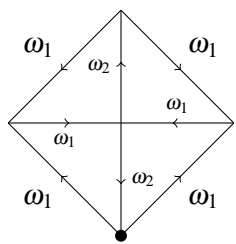

A {\em geodesic} in a diskoid is a path between two vertices of the diskoid with minimal total weight, or length. Since this length is a sum of dominant minuscule weights, it lies in the dominant Weyl chamber and the comparison between two lengths is by the usual ordering of weights. A geodesic thus corresponds to a minimal cut path of the web $w$. Since the weights of two geodesics between a pair of vertices need not be comparable, in general there need not be a unique weight for all geodesics between two points. We say that the pair $(a,b)$ has {\em coherent geodesics} if all geodesics between $a$ and $b$ have the same length. Let $\d(a,b)$ be this length. A dual diskoid has coherent geodesics at a vertex $a$ if all pairs $(a,b)$ for vertices $b$ have coherent geodesics. Finally, a web is said to be {\em coherent} if it satisfies the following 3 conditions
\begin{enumerate}
\item The dual has coherent geodesics at the marked vertex.
\item Every internal vertex is contained along some geodesic between the marked vertex and the boundary.
\item For every edge $a\stackrel{\gamma}{\rightarrow}b$ the condition $\d(\bullet,b)-\d(\bullet,a)\in W\gamma$.
\end{enumerate}

Here $\bullet$ denotes the marked vertex of the dual diskoid or the face with the marked boundary point in the web. To any coherent web with boundary $\vlambda$, we can associate $\vmu\in P_\vlambda$ by taking $\mu_i=\d(\bullet,v_i)$ where $v_i$ are the external vertices of the dual diskoid. In \cite{Kuperberg:spiders}, Kuperberg defines the set of non-elliptic webs for $\SL[3]$, those webs with no internal faces of less than degree $6$. It follows from \cite{fkk} that non-elliptic webs are coherent.

The following lemma on geodesics will be useful:
\begin{lemma}\label{lem:geod} Any subpath of a geodesic is still a geodesic. If the pair $(a,b)$ has coherent geodesics, than any pair $(c,d)$ of vertices lying on on a geodesic of $(a,b)$ also has coherent geodesics.
\end{lemma}

\begin{proof} Let $c$ and $d$ be two vertices lying along a geodesic $\gamma$. If there exists a path $\delta$ between $c$ and $d$ with lower total weight than the section of $\gamma$ between $c$ and $d$, replacing that section of $\gamma$ with $\delta$ would result in a smaller path. But $\gamma$ is already a geodesic. 

For the other statement, suppose that $\gamma$ is a geodesic between a pair of vertices $(a,b)$ with coherent geodesics and $\delta$ any path between $c$ and $d$. Replacing the segment of $\gamma$ between $c$ and $d$ with $\delta$ results in a path with a total weight greater than or equal to that of $\gamma$ since the pair $(a,b)$ has coherent geodesics.
\end{proof}

\subsection{The Affine Grassmannian and Building}

Recall that the {\em affine Grassmannian} for $\PGL$ is defined as $\Gr=\PGL(\cK)/\PGL(\cO)$ and that the action of $\PGL(\cO)$ on the left splits $\Gr$ up into locally closed orbits $\Gr_\lambda$ generated $t^\lambda$ where $\lambda$ is a dominant weight. Then, the left diagonal $\PGL(\cK)$ orbits on $\Gr\times\Gr$ are in one-to-one correspondence with the orbits $\Gr_\lambda$. Thus for any two points $L,L'\in\Gr$ we let $\d(L,L')$ be the label of the orbit containing $(L,L')$. This weight valued distance can be thought of as the relative position of the $\PGL(\cO)$ lattices given by the columns of $L$ and $L'$.

The points of $\Gr$ are also the set of vertices of the {\em affine building} $\Delta=\Delta(\PGL)$. This is a simplicial complex with whose $i$-simplicies are $(L_1,...,L_i)$ when $\d(L_j,L_k)$ is minuscule for all $j,k$ and $i<n$. This complex satisfies the following axioms:
\begin{enumerate}
\item The building is a non-disjoint union of {\em apartments} $\Sigma$ each of which is a copy of the Weyl alcove simplicial complex for $\PGL$.
\item Any two simplicies are contained in at least one apartment.
\item Given given any two apartments $\Sigma$ and $\Sigma'$ and simplicies $\alpha,\alpha'\in \Sigma\cap\Sigma'$, we have an isomorphism $f:\Sigma\rightarrow\Sigma'$ that fixes $\alpha$ and $\alpha'$.
\end{enumerate}

The affine building has coherent geodesics and the resulting distance between two vertices agrees with notion of distance coming from $\Gr$. As a result, we have the following:

\begin{lemma} If $p,q$ are simplicies of $\delta$, then any apartment $\Sigma$ containing $p$ and $q$ also contains every geodesic between them. 
\end{lemma}

In other words, any apartment containing $p$ and $q$ also contains the convex hull of $p$ and $q$.

Let $A$ be a directed graph labelled by dominant weights with a distinguished vertex $\bullet$. Then the {\em configuration space} $Q(A)$ is the variety of maps $f:v(A)\rightarrow\Gr$ (here $v(A)$ is the vertex set of $A$) satisfying the following conditions: $f(\bullet)=t^0$ and $\d(f(v),f(v'))=\d(v,v')$ for every edge $v\stackrel{\d(v,v')}{\rightarrow} v'$. In this setting is is natural to define the configuration space of a polyline $L(\vlambda)$ labelled by a sequence of dominant minuscule weights $\vlambda.$ $$Q(L(\vlambda))=\Gr_\vlambda=\{(L_0,\cdots,L_n)\in\Gr^{n+1}|L_0=t^0,\d(L_{i-1},L_i)=\lambda_i\}.$$ The set $\Gr_\vlambda$ is known as the twisted product of orbits and the projection on the last vertex is the {\em convolution} (or {\em multiplication}) morphism $m_\vlambda:\Gr_\vlambda\rightarrow\Gr$. The fibre, $F(\vlambda)=m^{-1}(0)$ is a projective variety generally called the {\em Satake fibre} but in the case of $\SL$, it is also isomorphic to a Grothendieck-Springer fibre. Note that $F(\vlambda)$ can also be realized as the configuration space of a closed polygon with edges labelled $\vlambda$. Then we have the following corollary of the Geometric Satake isomorphism:

\begin{corollary} $$\Hom_{\SL}(\C,V_\vlambda)\cong H_\top(F(\vlambda)),$$ where $H_\top$ is the top homology of $F(\vlambda)$. 
\end{corollary}

The top dimensional components $Z$ of $F(\vlambda)$ are a basis for $H_\top(F(\vlambda))$ and thus give a basis $[Z]\in \Hom_{\SL}(\C,V_\vlambda)$ called the {\em Satake basis}.

If $\vlambda$ is a dominant minuscule sequence of weights, then a {\em minuscule Littelmann path} $\vmu$ of type $\vlambda$ (ending at $0$) is a sequence of dominant weights such that $\mu_0=\mu_n=0$ and $\mu_i-\mu_{i-1}\in W\lambda_i$ where $W$ is the Weyl group of $\SL$. Let $P_\vlambda$ be the set of these paths. Then the top dimensional components of $F(\vlambda)$ are in one-to-one correspondence with the elements of $P_\vlambda$ and are constructed in the following manner:

\begin{theorem} \cite{fkk} Given $\vmu\in P_\vlambda$, let $A(\vlambda,\vmu)$ be the fan graph:
$$
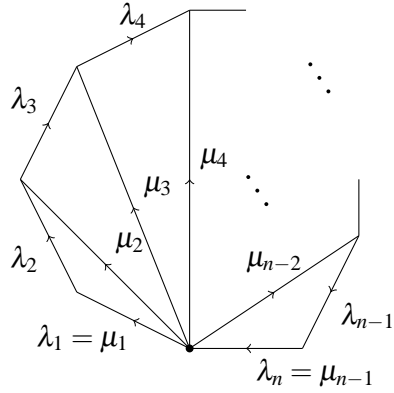
\begin{tikzpicture}[scale=.75,ar/.style={postaction={decorate,decoration={markings,mark=at position .5 with {\arrow{>}}}}}]
\draw[ar] (0,0) to node[below left=-.5ex] {$\lambda_1=\mu_1$} (-2,1);
\draw[ar] (-2,1) to node[below left] {$\lambda_2$} (-3,3);
\draw[ar] (-3,3) to node[above left] {$\lambda_3$} (-2,5);
\draw[ar] (-2,5) to node[above] {$\lambda_4$} (0,6);
\draw (0,6) -- (1,6);
\draw (3,3) -- (3,2);
\draw[ar] (3,2) to node[below right] {$\lambda_{n-1}$} (2,0);
\draw[ar] (2,0) to node[below right] {$\lambda_n = \mu_{n-1}$} (0,0);
\draw[ar] (0,0) to node[above right] {$\mu_2$} (-3,3);
\draw[ar] (0,0) to node[above right] {$\mu_3$} (-2,5);
\draw[ar] (0,0) to node[above right] {$\mu_4$} (0,6);
\draw[ar] (0,0) to node[above=.75ex] {$\mu_{n-2}$} (3,2);
\fill (2.14,5.04) circle (.035); \fill (2.3,4.8) circle (.035);
\fill (2.46,4.56) circle (.035);
\fill (1.04,3.04) circle (.035); \fill (1.2,2.8) circle (.035);
\fill (1.36,2.56) circle (.035);
\fill (0,0) circle(2pt);
\end{tikzpicture}
$$
Then $Q(A(\vlambda,\vmu))\subset F(\vlambda)$, $\overline{Q(A(\vlambda,\vmu))}$ is an irreducible component of $F(\vlambda)$ and each irreducible component appears in this way.
\end{theorem}

Since the dual diskoid is a direct graph labelled by dominant minuscule weights with a distinguished vertex, we can consider the configuration space $Q(D(w))$. If $v_0=\bullet,\cdots,v_n=\bullet$ are the exterior vertices of $D(w)$ read clockwise, define the boundary map $\pi:Q(D(w))\rightarrow F(\vlambda)$ by taking $\pi(f)=(f(v_0),\cdots,f(v_n)).$ Then we have the following results from \cite{fkk}:

\begin{theorem} There exists a class $c(w)\in H_\top(Q(D(w)))$ such that $\pi_*(c(w))=\Psi(w)$.
\end{theorem}

\begin{corollary} $\Psi(w)$ is a linear combination of the classes $[Q(A(\vlambda,\vmu))]$ which occur in the image of $\pi$.
\end{corollary}

Given $\vmu,\vmu'\in P_\vlambda$, we can define a partial ordering: $\vmu'\leq\vmu$ iff $\vmu'_i\leq\vmu_i$ for all $i$.

\begin{theorem} Given a coherent web $w$ with boundary $\vlambda$ and associated path $\vmu\in P_\vlambda$ the following holds:
\begin{enumerate}
\item Suppose that $\vmu'\in P_\vlambda$ and $\vmu'\not\leq\vmu$, then the coefficient of $[\overline{Q(A(\vlambda,\vmu)}]$ in $\Psi(w)$ is $0$.
\item The coefficient of $[\overline{Q(A(\vlambda,\vmu)}]$ in $\Psi(w)$ is $1$.
\end{enumerate}\label{thm:expansion}
\end{theorem}

The first part of this theorem follows directly from the previous result, since $Q(A(\lambda,\vmu'))$ can only be in the image of $\pi$ for a web $w$ if the associated path $\vmu$ satisfies $\vmu'\leq\vmu$. The second part is proved in \cite{fkk}. It follows from the following lemma and the fact that the scheme $\pi^{-1}(Q(A(\vlambda,\vmu)))$ is isomorphic to the reduced scheme $Q(A(\vlambda,\vmu))$.

\begin{lemma} Suppose that $w$ is a coherent web, then any configuration $f\in Q(A(\vlambda,\vmu))$ extends uniquely to a configuration $f\in Q(D(w))$.
\end{lemma}

\begin{proof} Let $f\in Q(A(\lambda,\vmu))$. Since the boundary cycles of $A(\vlambda,\vmu)$ and $D(w)$ are the same, we only need extend $f$ to the interior vertices. Let $v_i$ be the exterior vertices of $D(w)$, then $\d(f(\bullet),f(v_i))=\mu_i$ by definition and $\mu_i$ is the length of any geodesic $\gamma$ between $\bullet$ and $v_i$. Considering $\Gr$ as the vertices of associated affine building, let $\Sigma$ be an apartment containing $f(\bullet)=t^0$ and $f(v_i)$. We can choose coordinates on $\Sigma$ such that $f(v_i)$ has position $\mu_i$ and $t^0$ has position $0$. It follows that there is a unique geodesic $\tilde{\gamma}\in\Sigma$ between $t^0$ and $f(v_i)$ with the same sequence of edge weights as $\gamma$. Thus we can extend $f$ by setting $f(\gamma_i)=\tilde{\gamma}_i$. Since $w$ is coherent, by condition $2$ in the definition of coherence it follows that $f$ is now defined for any vertex of $D(w)$. 

Next we must show that this extension is consistent on vertices $v$ in the interior of $D(w)$. Suppose that we have two geodesics $\gamma$ and $\gamma'$ from $\bullet$ to $v_i$ and $v_j$ for some $i<j$ such that $v\in\gamma\cap\gamma'$. By induction on $j-i$, we will show that $f(v)$ is well defined. In the case that $j-i\leq 1$, this is obvious since there exists an apartment $\Sigma$ containing $t^0$, $f(v_i)$ and $f(v_j)$ since $\d(f(v_i),d(v_j))$ is either zero or minuscule, so $\{f(v_i),f(v_j)\}$ is a 0 or 1 simplex in $\Delta$. Since $\Sigma$ contains the convex hull of $t^0$, $f(v_i)$ and $f(v_j)$, both $\gamma$ and $\gamma'$ lie in $\Sigma$ so they both define $f(v)$ in the same way. In the case that $j-i>1$, consider any geodesic $\gamma''$ from $t^0$ to $f(v_{i+1})$. Since the paths $\gamma$ from $v_i$ to $v$ and $\gamma'$ from $v_j$ to $v$ form a closed region with the boundary of the web, it follows that $\gamma''$ intersects either $\gamma$ or $\gamma'$ at some point between $v$ and $v_i$ or $v$ and $v_j$. Without loss of generality, assume that $\gamma$ and $\gamma''$ intersect at some point $v'$ after $v$. Then the path along $\gamma$ from $\bullet$ to $v'$ and then from $v'$ to $v_{i+1}$ is a geodesic passes through $v$, and so by induction we are done.

Since $f$ is now well defined, we need to check that $f$ is indeed a configuration. That is, we must show that $$\d(f(a),f(b))=\d(a,b)$$ for $a$, $b$ adjacent vertices. Without loss of generality, there exists geodesics $\gamma$ and $\gamma'$ from $\bullet$ to $v_i$ and $v_j$ for some $i\leq j$ such that $a\in\gamma$ and $b\in\gamma'$. If $j-i\leq 1$, then there exists an apartment $\Sigma$ containing $f(\bullet)$, $f(v_i)$ and $f(v_j)$ as above and $\Sigma$ contains both geodesics $\gamma$ and $\gamma'$. There are coordinates on this apartment such that $t^0$ has position $0$, $f(v_i)$ has position $\d(\bullet,v_i)$ and $f(v_j)$ has position $\d(\bullet,v_j)$. By construction $f(a)$ has position $\d(\bullet,a)$ and $f(b)$ has position $\d(\bullet,b)$, but then since $w$ is coherent $\d(\bullet,b)-\d(\bullet,a)\in W\d(a,b)$, $f(a)$ and $f(b)$ have adjacent positions. Thus it follows that $\d(f(a),f(b))=\d(a,b)$ as needed. If $j-i>1$ consider any geodesic $\gamma''$ from $\bullet$ to $v_{i+1}$. As before it intersects either $\gamma$ in between $v_i$ and $a$ or $\gamma'$ in between $v_j$ and $b$. In either case, we can replace one of $v_i$ or $v_j$ with $v_{i+1}$ and a suitable geodesic formed from $\gamma''$ and $\gamma$ or $\gamma'$. By induction we see that $f\in Q(D(w))$.

Finally, this extension is unique since there was a unique choice of $f(v)$ for any internal vertex.
\end{proof}

As a corollary, we have the following result:

\begin{corollary}\label{cor:basis} Let $\vlambda$ be a sequence of dominant minuscule weights. Let $W$ be a set of coherent webs in one-to-one correspondence with $\vmu\in P_\vlambda$ such that $w\in W$ has a distinct associated path $\vmu\in P_\vlambda$. Then $\Psi(W)$ forms a basis for $\Hom_{\SL}(\C,V_\vlambda)$ and the change of basis from $\Psi(W)$ to the Satake basis is upper unitriangular with respect to the partial ordering of $P_\vlambda$.
\end{corollary}

\section{Triangular Diagrams}

Our goal now is to produce a set of coherent webs and to do this we will adapt Westbury's method of triangular diagrams. In \cite{Westbury:triangles}, Westbury introduces the idea of a triangular diagram. He defines a {\em triangular diagram} as a web drawn in a triangle rather than a disk, with the vertices of the triangle boundary labelled $A$, $X$ and $Y$. Then $A$, $X$ and $Y$ correspond to vertices in the dual diskoid and the edges $AY$ and $AX$ must be geodesics when considered as paths in the dual diskoid. Westbury also defines the notion of an {\em irreducible triangular diagram}: a triangular diagram is said to be irreducible if there exists no face $B$ such that there exists geodesics in the dual diskoid from $A$ to $X$ and $A$ to $Y$ both passing through the vertex corresponding to $B$. The length of a triangular diagram is the number of web edges incident to the section of the boundary $XY$.

Since we want to apply the ideas of the previous section to webs, it is natural to extend this definition to include the condition that the dual diskoid has coherent geodesics at $A$. We consider the point $A$ to be the marked point of the web and is called the vertex of the triangular diagram. The face containing $A$ is the exterior face and the faces along the boundary $XY$ are the topmost faces.

In order to produce webs, Westbury constructs irreducible diagrams of length $1$ and defines a growth algorithm to calculate the product of triangular diagrams. In the next two sections we will generalize these two ideas to the $\SL$ case and then show that the resultant webs are coherent. Constructing one web associated to each $\vmu\in P_\vlambda$ then results in a basis by corollary~\ref{cor:basis}.

\subsection{Irreducible diagrams of length one}

For each weight in the Weyl orbit of a dominant minuscule weight we must produce an irreducible triangular diagram of length one. These will be our building blocks to construct diagrams with larger lengths. Recall that if we consider the weights of $\SL$ to be the set $\mathbb{Z}^n/(1,1,\cdots,1)$, a weight that lies in the Weyl orbit of dominant minuscule weight $\omega_k$ is a sequence of $k$ ones and $n-k$ zeroes. Let $\omega$ be such a weight, then $\omega$ can be written uniquely as a sum of fundamental weights with coefficients in $\{-1,0,1\}$ in the following way: let $\omega=(c_1,c_2,\cdots,c_n)$, with $c_i\in\{0,1\}$ then $\omega=\sum_{j} (c_j-c_{j+1})\omega_j$. Let $l_1<l_2<\cdots$ be the sequence of $j$ such that $c_j-c_{j+1}=-1$ and $r_1<r_2<\cdots$ be the sequence of $j$ such that $c_j-c_{j+1}=1$. Then we have either $r_1<l_1<r_2<l_2<\cdots$ or $l_1<r_1<l_2<r_2<\cdots$. We will produce the triangular diagram for the first case, the second is similar. For convenience, edges of a web will be labelled with integers $0<i<n$ corresponding to a label $\omega_i$ or sometimes $0$, in which case the edge can be considered an edge of weight $0$ which does not exist in the corresponding web. The condition on the edge labels at each vertex in a web becomes the condition that the sum of the labels of the incoming edges must be equal to the sum of the outgoing edges modulo $n$. 

\begin{theorem}
The following is a irreducible triangular diagram which we denote $T_\omega$:

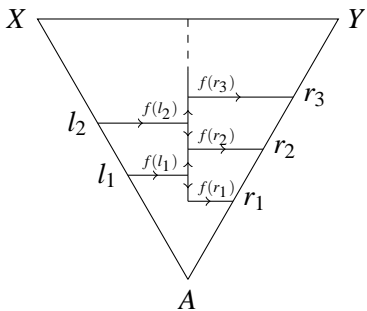
$$
\begin{tikzpicture}[ar/.style={postaction={decorate,decoration={markings,mark=at position .5 with {\arrow{>}}}}}]
\coordinate (b) at (0,0);
\draw (b) node[below] {$A$} -- (60:4) node[right] {$Y$} -- (120:4) node[left] {$X$}-- (b);
\coordinate (m) at ($ (60:4)!0.5!(120:4) $);
\draw[dashed] (m) -- ($ (m)!0.2!(b) $);
\draw ($ (m)!0.2!(b) $) -- ($ (m)!0.3!(b) $) node[right,midway] {${\scriptscriptstyle f(r_3)}$};
\draw [ar] ($ (m)!0.4!(b) $) -- ($ (m)!0.3!(b) $) node[left,midway] {${\scriptscriptstyle f(l_2)}$};
\draw [ar] ($ (m)!0.4!(b) $) -- ($ (m)!0.5!(b) $) node[right,midway] {${\scriptscriptstyle f(r_2)}$};
\draw [ar] ($ (m)!0.6!(b) $) -- ($ (m)!0.5!(b) $) node[left,midway] {${\scriptscriptstyle f(l_1)}$};
\draw [ar] ($ (m)!0.6!(b) $) -- ($ (m)!0.7!(b) $) node[right,midway] {${\scriptscriptstyle f(r_1)}$};
\draw [ar] ($ (m)!0.3!(b) $) -- ($ (60:4)!0.3!(b) $) node[right] {$r_3$};
\draw [ar] ($ (120:4)!0.4!(b) $) node[left] {$l_2$} -- ($ (m)!0.4!(b) $);
\draw [ar] ($ (m)!0.5!(b) $) -- ($ (60:4)!0.5!(b) $) node[right] {$r_2$};
\draw [ar] ($ (120:4)!0.6!(b) $) node[left] {$l_1$} -- ($ (m)!0.6!(b) $);
\draw [ar] ($ (m)!0.7!(b) $) -- ($ (60:4)!0.7!(b) $) node[right] {$r_1$};
\end{tikzpicture}$$\label{thm:lengthone}
\end{theorem}

It should be noted that the diagram above is by no means the only irreducible length one diagram corresponding to the given weight. In general there is at least one other diagram: the proof below follows if all inequalities and order of the edge labels are reversed. Note that there may be no irreducible diagram corresponding to an arbitrary ordering of the labels $l_i$ and $r_i$. For $n\leq 3$ this ambiguity does not appear since the lengths of the sequences $l_i$ and $r_i$ are at most $1$. In section~\ref{sec:sl4}, it will be shown that in the $\SL[4]$ case it is possible to minimize the number of vertices in the triangular diagram by carefully choosing which diagram should be used for $T_\omega$.

\begin{proof}
The edges from the left to the middle are labelled $l_i$, those from the middle to the right $r_i$ and the interior edges are labelled $f(r_i)$ or $f(l_i)$ with $f(r_1)=r_1$. At each vertex the sum of the incoming edges equals the sum of the outgoing edges mod $n$. This uniquely specifies the $f(r_i)$ and $f(l_i)$ and we will show that $0<f(r_i)\leq r_i$ and $0<f(l_i)\leq l_i$ for all $i$. Clearly this is true for $i=1$, since $0<f(r_1)=r_1$ and then since $f(l_1)=l_1-f(r_1)$ and $l_1>f(r_1)=r_1>0$, we have $0<f(l_1)<l_1$. By induction, it is true for all $i$: if $0<f(r_{i-1})\leq r_{i-1}$ and $0<f(l_{i-1})\leq l_{i-1}$, then we have $f(r_i)=r_i-f(l_{i-1})$ and since $r_i>l_{i-1}\geq f(l_{i-1})>0$, we have $0<f(r_i)<r_i$ and similarly $0<f(l_i)<l_i$. Note that since none of the edges are labelled $0$, they all appear in the web.

To establish that this is a triangular diagram we must show that $A$ has coherent geodesics and that $AX$ and $AY$ are geodesics. Consider figure~\ref{fig:lengthonedual}, the dual diskoid for the above diagram.
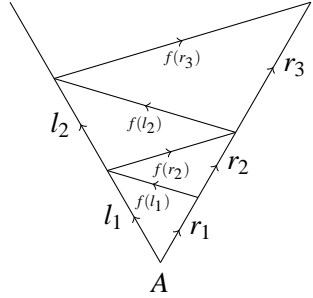
\begin{figure*}
\begin{tikzpicture}[ar/.style={postaction={decorate,decoration={markings,mark=at position .5 with {\arrow{>}}}}}]
\draw [ar] (0,0) node[below] {$A$} -- (60:1) node[midway,right] {$r_1$};
\draw [ar] (60:1) -- (60:2) node[midway,right] {$r_2$};
\draw [ar] (60:2) -- (60:4) node[midway,right] {$r_3$};
\draw [ar] (0,0) -- (120:1.414) node[midway,left] {$l_1$};
\draw [ar] (120:1.414) -- (120:2.828) node[midway,left] {$l_2$}; 
\draw (120:2.828) -- (120:4);
\draw [ar] (60:1) -- (120:1.414) node[midway,below] {${\scriptscriptstyle f(l_1)}$};
\draw [ar] (120:1.414) -- (60:2) node[midway,below] {${\scriptscriptstyle f(r_2)}$};
\draw [ar] (60:2) -- (120:2.828) node[midway,below] {${\scriptscriptstyle f(l_2)}$};
\draw [ar] (120:2.828) -- (60:4) node[midway,below] {${\scriptscriptstyle f(r_3)}$};
\end{tikzpicture}
\caption{The dual diskoid of a triangle of length $1$.}
\label{fig:lengthonedual}
\end{figure*}
By Lemma~\ref{lem:geod}, since the diagram has length $1$ and there are no internal vertices, it suffices to show that $AX$ and $AY$ are geodesics and that the pairs $(A, X)$ and $(A, Y)$ have coherent geodesics.

Suppose we have a path $\gamma$ between $A$ and $X$ in the dual diskoid. Then we want to show that the total weight of $\gamma$ is greater than the total weight of $AX$. We will apply a series of operations to $\gamma$ that reduce its weight and result in the path $AX$. If $\gamma$ has any loops, remove them, this lowers the total weight of $\gamma$. Suppose that the path $\gamma$ travels along two edges of a triangular face:
$$
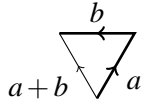
\begin{tikzpicture}
\draw [line width=1pt,postaction={decorate,decoration={markings,mark=at position .25 with {\arrow{>}},mark=at position .75 with {\arrow{>}}}}] (0,0) -- (60:1) node[midway,below right] {$a$} -- (120:1) node[midway,above] {$b$};
\draw [postaction={decorate,decoration={markings,mark=at position .5 with {\arrow{>}}}}] (0,0) -- (120:1) node[midway,below left] {$a+b$};
\end{tikzpicture}$$
Then using the edge labelled $a+b$ rather than the two edges labelled $a$ and $b$ results in lowering the weight of that segment of the path from $\omega_a+\omega_b$ to $\omega_{a+b}$ where $a+b$ is reduced modulo $n$. These are called triangle moves in \cite{fkk}. Thus we can assume that $\gamma$ has no loops or available triangle moves. Thus if $\gamma$ is not already $AX$, it travels some number of edges up the left side (possibly $0$) then crosses to the right. It must cross on an edge labelled $f(r_i)$ otherwise the path would contain a loop or triangle move. Then the path cannot travel along $f(l_i)$ or $r_i$ otherwise there would also be a triangle move. Thus the path travels along $r_{i+1}$. So the path travels along the the left side of the following diamond:
$$
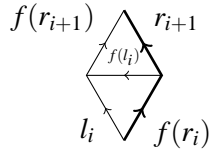
\begin{tikzpicture}
\draw [line width=1pt,postaction={decorate,decoration={markings,mark=at position .25 with {\arrow{>}},mark=at position .75 with {\arrow{>}}}}] (0,0) -- (60:1) node[midway,below right] {$f(r_i)$} -- (90:1.732) node[midway,above right] {$r_{i+1}$};
\draw [postaction={decorate,decoration={markings,mark=at position .25 with {\arrow{>}},mark=at position .75 with {\arrow{>}}}}] (0,0) -- (120:1) node[midway,below left] {$l_i$} -- (90:1.732) node[midway,above left] {$f(r_{i+1})$};
\draw [postaction={decorate,decoration={markings,mark=at position .5 with {\arrow{>}}}}] (60:1) -- (120:1) node[midway,above] {${\scriptscriptstyle f(l_i)}$};
\end{tikzpicture}$$
By doing a diamond move, and swapping the path to the other side, $\gamma$ would travel further along the left side. For this move to lower the weight of the path, we must have $\omega_{f(r_i)}+\omega_{r_{i+1}}\geq\omega_{l_i}+\omega_{f(r_{i+1})}$. Since $l_i-f(r_i)=f(l_1)=r_{i+1}-f(r_{i+1})$ we have $f(r_i)+r_{i+1}=l_i+f(r_{i+1})$ so the weights differ by a number of roots. We have $r_{i+1}>l_i>r_i\geq f(r_i)$ so $r_{i+1}>l_i>f(r_i)$ and thus $r_{i+1}>f(r_{i+1})>f(r_i)$. But this means that we have $\omega_{f(r_i)}+\omega_{r_{i+1}}\geq\omega_{l_i}+\omega_{f(r_{i+1})}$ as needed. Thus by a sequence of removing loops, triangle and diamond moves we can reduce to path so that it only travels along $AX$ and thus $AX$ is a geodesic and the pair $(A, X)$ has coherent geodesics.

The proof for $AY$ is similar.
\end{proof}

\subsection{Product of Triangular Diagrams}

Given two triangular diagrams $A$ and $B$ we can define their product $A\otimes B$. Placing the diagrams for $A$ and $B$ side by side as in Figure~\ref{fig:product}, we can fill the diamond $C$ so that the resultant diagram is a triangular diagram.

\begin{figure}[!ht]
\begin{tikzpicture}
\draw (0,0) -- (60:4) -- (120:4) -- cycle;
\draw (60:2) -- ($ (60:2) + (120:2) $) -- (120:2);
\draw ($ (60:2.66) + (120:0.66) $) node {$B$};
\draw ($ (60:0.66) + (120:2.66) $) node {$A$};
\draw ($ (60:1) + (120:1) $) node {$C$};
\end{tikzpicture}\caption{Product of Triangular Diagrams $A$ and $B$ via the diamond $C$}\label{fig:product}
\end{figure}
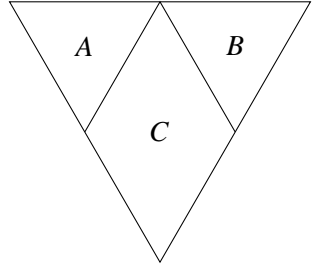

Let $i_1,\cdots,i_l$ be the labels of the (outgoing) edges on the right side of triangle $A$ from bottom to top and $i'_1,\cdots,i'_{l'}$ the labels on the (incoming) edges on the left side of $B$. Then the diamond $C$ is defined to be the result of taking the diamond: $$
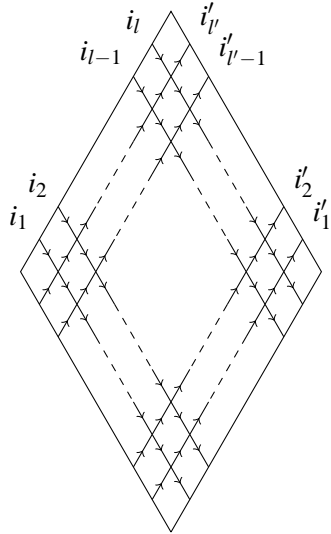
\begin{tikzpicture}[ar/.style={postaction={decorate,decoration={markings,mark=at position .5 with {\arrow{>}}}}}]
\draw (0,0) -- (60:4) -- +(120:4) -- (120:4) -- (0,0);
\foreach \ind in {0.5,1,3,3.5} {
	\foreach \ing in {4,3.5,3,1.5,1,0.5} {
		\draw[ar] ($ (120:\ing) + (60:\ind) $) -- ($ (120:\ing) - (120:0.5) + (60:\ind) $);
		\draw[ar] ($ (60:\ing) - (60:0.5) + (120:\ind) $) -- ($ (60:\ing) + (120:\ind) $);
	}
	\draw[dashed] ($ (120:2.5) + (60:\ind) $) -- ($ (120:1.5) + (60:\ind) $);
	\draw[dashed] ($ (60:2.5) + (120:\ind) $) -- ($ (60:1.5) + (120:\ind) $);
	}
	\draw ($ (120:4) + (60:0.5) $) node[above left] {$i_1$};
	\draw ($ (120:4) + (60:1) $) node[above left] {$i_2$};
	\draw ($ (120:4) + (60:3) $) node[above left] {$i_{l-1}$};
	\draw ($ (120:4) + (60:3.5) $) node[above left] {$i_l$};
	\draw ($ (60:4) + (120:0.5) $) node[above right] {$i'_1$};
	\draw ($ (60:4) + (120:1) $) node[above right] {$i'_2$};
	\draw ($ (60:4) + (120:3) $) node[above right] {$i'_{l'-1}$};
	\draw ($ (60:4) + (120:3.5) $) node[above right] {$i'_{l'}$};
\end{tikzpicture}$$ and replacing each vertex of degree 4 via the following rule (evaluated from the top most vertex to the bottom most):

\begin{equation*}\begin{split}
\begin{tikzpicture}[ar/.style={postaction={decorate,decoration={markings,mark=at position .5 with {\arrow{>}}}}},scale=0.6]
\draw[ar] (120:1) node[above left] {$a$} -- (0,0); \draw[ar] (0,0) -- (300:1);
\draw[ar] (240:1) -- (0,0); \draw[ar] (0,0) -- (60:1) node[above right] {$b$};
\end{tikzpicture}
\end{split} = 
\begin{cases}
\begin{split}\begin{tikzpicture}[ar/.style={postaction={decorate,decoration={markings,mark=at position .5 with {\arrow{>}}}}},scale=0.6]
\draw[ar] (120:1) node[above left] {$a$} -- (0,0.3); \draw[ar] (0,-0.3) -- (300:1) node[below right] {$0$};
\draw[ar] (240:1) node[below left] {$0$} -- (0,-0.3); \draw[ar] (0,0.3) -- (60:1) node[above right] {$a$};
\draw[ar] (0,-0.3) -- (0,0.3) node[midway,left] {$0$};
\end{tikzpicture}\end{split}
& \text{$a=b$}, \\
\begin{split}\begin{tikzpicture}[ar/.style={postaction={decorate,decoration={markings,mark=at position .5 with {\arrow{>}}}}},scale=0.6]
\draw[ar] (120:1) node[above left] {$a$} -- (0,0.3); \draw[ar] (0,-0.3) -- (300:1) node[below right] {$a$};
\draw[ar] (240:1) node[below left] {$b$} -- (0,-0.3); \draw[ar] (0,0.3) -- (60:1) node[above right] {$b$};
\draw[ar] (0,-0.3) -- (0,0.3) node[midway,left] {$b-a$};
\end{tikzpicture}\end{split}
& \text{$a\neq b$}.
\end{cases}
\end{equation*}

Note that as previously mentioned an edge with $0$ weight does not appear in the web, but in order to simplify arguments we will leave it in. Taking the dual diskoid still makes sense with edges labelled $0$, the corresponding web variety is unchanged since the we can contract edges of weight $0$ and delete the resulting $2$ faces to obtain the dual diskoid of the web. We can now define the left and right diamond moves. Consider a path in the dual diskoid of a diamond:
$$
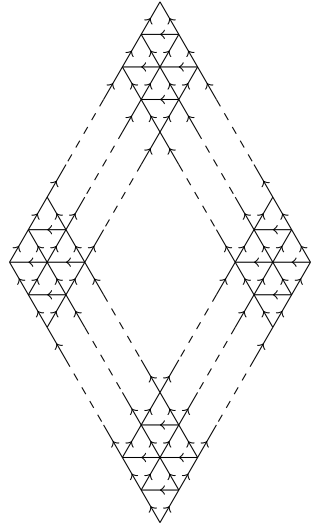
\begin{tikzpicture}[ar/.style={postaction={decorate,decoration={markings,mark=at position .5 with {\arrow{>}}}}}]
\foreach \ind in {0,0.5,1,3,3.5,4} {
	\foreach \ing in {4,3.5,3,1.5,1,0.5} {
		\draw[ar] ($ (120:\ing) - (120:0.5) + (60:\ind) $) -- ($ (120:\ing) + (60:\ind) $);
		\draw[ar] ($ (60:\ing) - (60:0.5) + (120:\ind) $) -- ($ (60:\ing) + (120:\ind) $);
	}
	\draw[dashed] ($ (120:2.5) + (60:\ind) $) -- ($ (120:1.5) + (60:\ind) $);
	\draw[dashed] ($ (60:2.5) + (120:\ind) $) -- ($ (60:1.5) + (120:\ind) $);
}
\draw[ar] (60:0.5) -- +(-0.5,0);
\draw[ar] (60:1) -- +(-0.5,0);
\draw[ar] ($ (60:1) - (0.5,0) $) -- +(-0.5,0);
\draw[ar] ($ (60:1.5) - (0.5,0) $) -- +(-0.5,0);
\draw[ar] (60:3.5) -- +(-0.5,0);
\draw[ar] ($ (60:3.5) - (3,0) $) -- +(-0.5,0);
\draw[ar] (60:4) -- +(-0.5,0);
\draw[ar] ($ (60:4) - (0.5,0) $) -- +(-0.5,0);
\draw[ar] ($ (60:4) - (3,0) $) -- +(-0.5,0);
\draw[ar] ($ (60:4) - (3.5,0) $) -- +(-0.5,0);
\draw[ar] ($ (60:4) + (120:0.5) $) -- +(-0.5,0);
\draw[ar] ($ (60:4) + (120:0.5) -(3,0) $) -- +(-0.5,0);
\draw[ar] ($ (60:3.5) + (120:3) $) -- +(-0.5,0);
\draw[ar] ($ (60:4) + (120:3) $) -- +(-0.5,0);
\draw[ar] ($ (60:4) + (120:3) -(0.5,0) $) -- +(-0.5,0);
\draw[ar] ($ (60:4) + (120:3.5) $) -- +(-0.5,0);
\end{tikzpicture}$$
The small diamonds inside are given by dual diskoid equivalent to the rule above: 
\begin{equation*} \begin{split}\begin{tikzpicture}[ar/.style={postaction={decorate,decoration={markings,mark=at position .5 with {\arrow{>}}}}}] 
\draw[ar] (0,0) -- (60:1) node[midway,below right] {$a$};
\draw[ar] (0,0) -- (120:1) node[midway,below left] {$b$};
\draw[ar] (60:1) -- ($ (60:1) + (120:1) $) node[midway,above right] {$b$};
\draw[ar] (120:1) -- ($ (60:1) + (120:1) $) node[midway,above left] {$a$};
\draw[ar] (60:1) -- (120:1) node[midway,above] {${\scriptscriptstyle b-a}$};
\end{tikzpicture}\end{split} \text{if $a\neq b$ or} \begin{split}
\begin{tikzpicture}[ar/.style={postaction={decorate,decoration={markings,mark=at position .5 with {\arrow{>}}}}}] 
\draw[ar] (0,0) -- (60:1) node[midway,below right] {$0$};
\draw[ar] (0,0) -- (120:1) node[midway,below left] {$0$};
\draw[ar] (60:1) -- ($ (60:1) + (120:1) $) node[midway,above right] {$a$};
\draw[ar] (120:1) -- ($ (60:1) + (120:1) $) node[midway,above left] {$a$};
\draw[ar] (60:1) -- (120:1) node[midway,above] {${\scriptscriptstyle 0}$};
\end{tikzpicture}\end{split} \text{otherwise}
\end{equation*}

Then, if a path travels along the right two edges of one of the above small diamonds, replacing that portion of the path with the left edges is a left diamond move and replacing the left edges with the right edges is a right diamond move. In either of the cases, doing a left or right $H$ move does not change the weight of the path.

\begin{proposition}\label{prop:triang} The product $A\otimes B$ is a triangular diagram.
\end{proposition}

\begin{proof} We will abuse notation and denote the vertex of $A$ by $A$, the vertex of $B$ by $B$ and the vertex of $A\otimes B$ as $C$: 
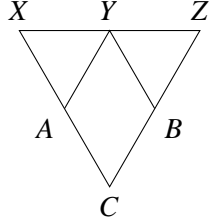
$$
\begin{tikzpicture}[scale=0.6]
\draw (0,0) node[below] {$C$} -- (60:4) node[above] {$Z$} -- (120:4) node[midway,above] {$Y$} node[above] {$X$} -- cycle;
\draw (60:2) node[below right] {$B$} -- ($ (60:2) + (120:2) $) -- (120:2) node[below left] {$A$};
\end{tikzpicture}$$
We will show that any path $\gamma$ in the dual diskoid starting at $C$ and ending at a vertex corresponding to a face in $A$ has weight greater than or equal to that of the path starting with $CA$ and then following a geodesic in the dual diskoid of $A$ to the endpoint. First, assume that the given path never intersects itself at any point during the proof, otherwise we would trim the loop in the path to obtain a path with lower weight.

Since the path $\gamma$ must end at a vertex in $A$, it follows that the path must cross $AY$. $AY$ is a geodesic and has the same weight in both the triangle $A$ and the diamond $C$, thus $AY$ is a geodesic in $A\otimes B$. If $\gamma$ crosses $AY$ more than once, then by Lemma~\ref{lem:geod} we can replace the portion of $\gamma$ starting from the first intersection with $AY$ to the last with the corresponding portion of the path $AY$, reducing the weight of $\gamma$. Thus we may assume that $\gamma$ starts at the vertex of the product, eventually enters triangle $A$ and never enters the diamond $C$ again.

Now we consider the segment of the path in the diamond and apply all possible left diamond and triangle moves to this path. Note that is a finite process and never enters a cycle since triangle moves reduce number of edges in the path and left diamond moves result in a path strictly to the left of the previous path. Since left diamond moves maintain the weight and triangle moves drop the weight, it follows that the resulting path has weight less than or equal to that of the original path. Thus any path in $C$ has weight greater than or equal to one between the same vertices with no left diamond moves or triangle moves. Such a path must follow $CA$ and then travel some distance along $AY$.

Consider the diamond $C$ in the oriented as above in the plane. Suppose first coordinate function $\pi_1$ function has local maxima, this can only happen at a vertex and thus $\gamma$ has one of the following three forms at that vertex:
$$
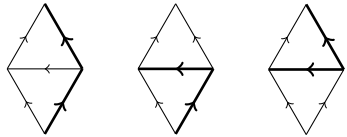
\begin{tikzpicture}[ar/.style={postaction={decorate,decoration={markings,mark=at position .5 with {\arrow{>}}}}}] 
\draw[ar,line width=1pt] (0,0) -- (60:1);
\draw[ar] (0,0) -- (120:1);
\draw[ar,line width=1pt] (60:1) -- ($ (60:1) + (120:1) $);
\draw[ar] (120:1) -- ($ (60:1) + (120:1) $);
\draw[ar] (60:1) -- (120:1);
\end{tikzpicture}
\hskip20pt
\begin{tikzpicture}[ar/.style={postaction={decorate,decoration={markings,mark=at position .5 with {\arrow{>}}}}}] 
\draw[ar,line width=1pt] (0,0) -- (60:1);
\draw[ar] (0,0) -- (120:1);
\draw[ar] (60:1) -- ($ (60:1) + (120:1) $);
\draw[ar] (120:1) -- ($ (60:1) + (120:1) $);
\draw[ar,line width=1pt] (60:1) -- (120:1);
\end{tikzpicture}
\hskip20pt
\begin{tikzpicture}[ar/.style={postaction={decorate,decoration={markings,mark=at position .5 with {\arrow{>}}}}}] 
\draw[ar] (0,0) -- (60:1);
\draw[ar] (0,0) -- (120:1);
\draw[ar,line width=1pt] (60:1) -- ($ (60:1) + (120:1) $);
\draw[ar] (120:1) -- ($ (60:1) + (120:1) $);
\draw[ar,line width=1pt] (60:1) -- (120:1);
\end{tikzpicture}
$$
In each of these situations, there is a triangle or left diamond move available. Thus $\pi_1$ has no local maxima and thus since $\gamma$ ends along $AY$, $\gamma$ travels along $CA$ and then it enters $A$ where it continues until it ends. The above argument holds for paths starting at $C$ and ending at a vertex in $B$ after replacing $A$ with $B$, left with right and maxima with minima. 

Now, to see the that the geodesics of $A\otimes B$ at $A$ are coherent, consider any geodesic $\gamma$ from $C$ to a vertex $D$ in $A$. Since $\gamma$ is a geodesic, by above argument $\gamma$ has the same weight as a path $\gamma'$ that follows $CA$ and then stays in $A$ until it reaches $D$. But then $\gamma'$ is a geodesic and by Lemma~\ref{lem:geod} the portion of $\gamma'$ from $A$ to $D$ is also a geodesic. But the weight of this geodesic is independent of the path because $A$ is a triangular diagram and itself has coherent geodesics. Thus the weight of $\gamma$ is the weight of $CA$ plus the weight of any geodesic from $A$ to $D$ and it does not depend on $\gamma$. The same holds for paths from $C$ to a vertex on $D$ in $B$ after replacing $A$ by $B$. It then follows that $CX$ and $CY$ are geodesics.
\end{proof}

\begin{proposition} The operator $\otimes$ on triangular diagrams is associative.
\end{proposition}

\begin{proof}
We must show that $(A\otimes B)\otimes C=A\otimes(B\otimes C)$ for any triangular diagrams $A,B,C$, but this is obvious.
\end{proof}

\subsection{Diagrams from Minuscule Paths}

In general, given a minuscule sequence $\vlambda$, a minuscule path of type $\vlambda$ is a sequence of weights $\vmu$ of $G$ such that $\mu_{i}-\mu_{i-1}\in W\lambda_i$ and $\mu_0=0$. Thus the set $P_\vlambda$ is the set of dominant minuscule path ending at $0$. To each minuscule path we can assign the triangular diagram $T_\vmu=\bigotimes T_{\mu_{i+1}-\mu_i}.$

\begin{lemma}\label{lem:sum}
Given a minuscule path $\vmu$ of length $m$, let $l_\vmu$ and $r_\vmu$ be the weights of the geodesics along the left and right sides of $T_\vmu$, we have $$\mu_m=r_\vmu-l_\vmu.$$
\end{lemma}

\begin{proof} This is true by construction if $m=1$, i.e. when $T_\vmu$ is an irreducible triangular diagram of length $1$. Otherwise consider the subsequence $\vmu'$ of the first $m-1$ elements of $\vmu$, we have $T_{\vmu'}\otimes T_{\mu_m-\mu_{m-1}}=T_\vmu$ and $T_{\vmu'}$ satisfies the lemma by induction on the length of the sequence. We have the following diagram:

$$
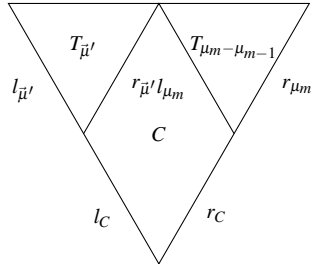
\begin{tikzpicture}
\draw (0,0) -- (60:2) node[midway,below right] {${\scriptstyle r_C}$} -- (60:4) node[midway,below right] {${\scriptstyle r_{\mu_m}}$} -- (120:4) -- (120:2) node[midway,below left] {${\scriptstyle l_{\vmu'}}$} -- (0,0) node[midway,below left] {${\scriptstyle l_C}$};
\draw (60:2) -- ($ (60:2) + (120:2) $) node[midway,below left] {${\scriptstyle l_{\mu_m}}$} -- (120:2) node[midway,below right] {${\scriptstyle r_{\vmu'}}$};
\draw ($ (60:2.66) + (120:0.66) $) node {${\scriptstyle T_{\mu_m-\mu_{m-1}}}$};
\draw ($ (60:0.66) + (120:2.66) $) node {${\scriptstyle T_{\vmu'}}$};
\draw ($ (60:1) + (120:1) $) node {${\scriptstyle C}$};
\end{tikzpicture}$$

Thus we have $\mu_{m-1}=r_{\vmu'}-l_{\vmu'}$ where $l_{\vmu'}$ and $r_{\vmu'}$ are the weights of the geodesics along the left and right sides of $T_{\vmu'}$. If $C$ is the diamond that appears in $T_{\vmu'}\otimes T_{\mu_m-\mu_{m-1}}$ then let $l_C$ and $r_C$ be the weights of the geodesics along the bottom left and bottom right sides of $C$ respectively. We have $l_{\vmu}=l_{\vmu'}+l_C$, $r_{\vmu}=r_{\mu_m}+r_C$ and also $\mu_m-\mu_{m-1}=r_{\mu_m}-l_{\mu_m}$. Thus 
\begin{align*}
r_\vmu-l_\vmu& =r_{\mu_m}+r_C-l_{\vmu'}-l_C\\
& =\mu_m-\mu_{m-1}+l_{\mu_m}+r_C-l_{\vmu'}-l_C \\
& =\mu_m-(r_{\vmu'}-l_{\vmu'})+l_{\mu_m}+r_C-l_{\vmu'}-l_C \\
& =\mu_m-r_{\vmu'}+l_{\mu_m}+r_C-l_C \\
& =\mu_m.
\end{align*}
since $l_{\mu_m}+r_C=r_{\vmu'}+l_C$ as they are weights of the geodesics along the left and right side of $C$.
\end{proof}

The next theorems show that for $\vmu\in P_\vlambda$, the triangular diagram $T_\vmu$ is a coherent web with associated path $\vmu$. To check that the associated path in $P_\vlambda$ is $\vmu$ it is sufficient to check that the weight of the geodesics along the left and right sides of the diagram are both $0$.

\begin{theorem} Let $\vmu$ be a dominant minuscule path, then the geodesic along the left side of $T_\vmu$ is of weight $0$.
\label{thm:dompathzero}\end{theorem}

\begin{proof} As usual, we proceed by induction on the length of the path. For a path of length 1, the first weight of the path is a dominant minuscule weight, i.e. it is $\omega_j$ for some $j$. The irreducible diagram for this weight has only one edge that exits on the right side of the triangle, proving the base case.

Suppose that we have some dominant minuscule path $\mu_i$ of length $m$, we will proceed using the conventions of Lemma \ref{lem:sum}. By induction $l_{\vmu'}=0$ so by Lemma \ref{lem:sum} we have $r_{\vmu'}=\mu_{m-1}$. Since $l_{\vmu'}=0$, what we must prove is $l_C=0$. Now, $r_{\vmu'}-l_{\mu_m}$ can be written as a difference of dominant weight which share no common fundamental weights, in fact by construction this decomposition is $r_C-l_C$. Thus if $r_{\vmu'}-l_{\mu_m}$ is dominant, $l_C=0$. Suppose that $r_{\vmu'}-l_{\mu_m}$ was not dominant, since $mu_m=\mu_{m-1}+(\mu_m-\mu_{m-1})=r_{\vmu'}+r_{\mu_m}-l_{\mu_m}$ is dominant, the set of fundamental weights that sum to $r_{\mu_m}$ would have to intersect those that sum to $l_{\mu_m}$, which is impossible by the construction of $T_{\mu_m-\mu_{m-1}}$.
\end{proof}

\begin{corollary} Any triangular diagram $T_\vmu$ with $\vmu\in P_\vlambda$ has no edges along the left or right side of the triangle.
\end{corollary}

\begin{proof} By the previous theorem, the left side has weight $0$ and $\vmu$ ends at $0$, so by Lemma~\ref{lem:sum}, the right side has weight $0$.
\end{proof}

\begin{theorem} For $\vmu\in P_\vlambda$, $T_\vmu$ is a coherent web.
\end{theorem}

\begin{proof} In fact, this will be true for any triangular diagram $T$ that is a product of the diagrams described in Theorem~\ref{thm:lengthone} or those with the opposite ordering of labels.

Since $T$ has coherent geodesics at its vertex, we must show that the other two criteria for a coherent web are satisfied. First we must show that any internal vertex of $T$ occurs along some geodesic between the vertex of $T$ and a vertex corresponding to a top most face. We proceed by induction. In the case of a single diagram of length one this is true by design, since all vertices of the dual lie on left or right sides of the diagram which are both geodesics. Using the notation of Proposition~\ref{prop:triang}, suppose that both $A$ and $B$ as above. For any vertex $D$ in $A$, then there exists a geodesic $\gamma'$ between $A$ and a vertex along the boundary $XY$ that passes through $D$. Then the path $\gamma=CA\sqcup\gamma'$ is a geodesic from $C$ to the boundary of the dual diskoid passing through $D$. Similarly we can construct a geodesic from $C$ to the boundary of the dual diskoid passing through any vertex of $B$. This leaves the internal vertices of the diamond. Consider the geodesic $CAY$, by right diamond moves it passes to the geodesic $CBY$ and any vertex in $C$ is on one of the intermediate geodesics. Thus any vertex in the dual diskoid of $A\otimes B$ lies on a geodesic between $C$ and a vertex on $XZ$.

Let $A$ be the vertex of $T$. The second condition that we need to check is that $$\d(A,b)-\d(A,a)\in W\d(a,b)$$ for adjacent vertices $a$ and $b$ where $A$ is the marked vertex of the dual diskoid. If the edge joining $a$ and $b$ is a part of a geodesic from $A$, then the statement is true since either $\d(A,b)+\d(b,a)=\d(A,a)$ or $\d(A,b)=\d(A,a)+\d(a,b)$. Otherwise we can break this into two cases: $a$ and $b$ are either both contained in a triangle $T'$ of length $1$ that appears in the construction of $T$ or not. In the first case, we can pick geodesics $\gamma_a$ and $\gamma_b$ from $A$ to $a$ and $A$ to $b$ that pass through the vertex $B$ of $T'$, then it follows that $\d(A,b)-\d(A,a)=\d(B,b)-\d(B,a)$. Since $a$ and $b$ are not joined by an edge contained in a geodesic from $A$ (or $B$), then following the construction of $T$ in figure~\ref{fig:lengthonedual}, we have $a$ appearing as the $i$-th vertex on the left and $b$ as the $j$-th vertex on the right with $|i-j|\leq 1$. Thus $\d(B,a)=\sum_{k=1}^i \omega_{l_k}$ and $\d(B,b)=\sum_{k=1}^j\omega_{r_k}$. Since the labels on the left and right are all unique, $$\d(B,b)-\d(B,a)=\sum_{k=1}^j\omega_{r_k}-\sum_{k=1}^i \omega_{l_k}\in W\omega_{\sum_{k=1}^j r_k-\sum_{k=1}^i l_k}.$$ But $\sum_{k=1}^j r_k-\sum_{k=1}^i l_k$ is the label of the edge between $a$ and $b$ by construction.

Otherwise, the edge appears in the dual diskoid of some diamond used in the construction of $T$. In any diamond, the only edges that are not part of a geodesic from $A$ are the horizontal edges. These appear in between the small diamonds outlined by the left and right diamond moves:
$$
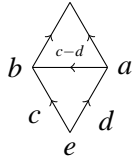
\begin{tikzpicture}[ar/.style={postaction={decorate,decoration={markings,mark=at position .5 with {\arrow{>}}}}}]
\draw[ar] (0,0) node[below] {$e$} -- (60:1) node[right] {$a$} node[midway, below right] {$d$};
\draw[ar] (0,0) -- (120:1) node[left] {$b$} node[midway, below left] {$c$};
\draw[ar] (60:1) -- (120:1) node[midway,above] {${\scriptscriptstyle c-d}$};
\draw[ar] (60:1) -- ($ (60:1) + (120:1) $);
\draw[ar] (120:1) -- ($ (60:1) + (120:1) $);
\end{tikzpicture}$$
There is a geodesic $\gamma$ from $A$ to $a$ passing though $e$ and $\gamma'$ from $A$ to $b$ passing through $e$. Thus $$d(A,b)-d(A,a)=(d(A,e)+\omega_c)-(d(A,e)+\omega_d)=\omega_c-\omega_d\in W\omega_{c-d}=W\d(a,b).$$ 
\end{proof}

Until this result the internal structure of the diagrams of length one was not important. The need for the resulting web to be coherent forces us to pick irreducible triangular diagrams of length one that satisfy the coherence conditions for webs. Both the diagram introduced in Theorem~\ref{thm:lengthone} and the similar diagram with the ordering of the labels switched satisfy these properties by construction. This completes the proof of Theorem~\ref{thm:main} since the webs coming from the triangular diagrams $T_\vmu$ form a set of basis webs by Corollary~\ref{cor:basis}.

\section{Webs as a basis for $\Hom_{\SL}(V_\nu,V_\vlambda)$}

In the previous section, given $\vmu\in P_\vlambda$ with length $k$, we constructed a corresponding web: $$\bigotimes_{i=0}^k T_{\mu_i-\mu_{i-1}}.$$ Now we consider the partial products $$\bigotimes_{i=0}^j T_{\mu_i-\mu_{i-1}}$$ for $j\leq k$. The underlying web is still coherent and thus results in a class in top homology. Since the irreducible representation $V_{\mu_j}$ of highest weight $\mu_j$ occurs in $V_{\lambda_1}\otimes\cdots\otimes V_{\lambda_j}$ with multiplicity vector space $$\Hom_{\SL}(V_{\mu_j},V_{\lambda_1}\otimes\cdots\otimes V_{\lambda_j})\cong H_\top(\overline{\Gr_{\mu_j}}\times_{\Gr}\overline{\Gr_{\lambda_1,\cdots,\lambda_j}})$$ it is reasonable to ask if the classes coming from these partial products form a basis for this space.

Denote the set of dominant minuscule paths of type $\vlambda$ ending at $\nu$ by $P_\vlambda(\nu)$. Using this notation $P_\vlambda=P_\vlambda(0)$. For $\vmu\in P_\vlambda(\nu)$, by theorem~\ref{thm:dompathzero}, the triangular diagram $T_\vmu$ has no incident edges on the left. Thus geodesic on the right has weight $\nu$ and results $\vnu$ the sequence of weights corresponding the the labels of the edges ending on the right side of the triangle from bottom to top. By the previous section, the web underlying $T_\vmu$ is confluent. Let $\tilde{\vmu}\in P_{\vlambda\sqcup\vnu^*}(0)$ be the associated path. Let $\phi:V_{\nu}\hookrightarrow V_{\vnu}$ be the canonical inclusion, then $\Psi_{\nu}=(id_{V_{\vlambda}}\otimes\phi^*)\circ \Psi$ is a map from webs of type $\vlambda\sqcup\vnu^*$ to $\Hom_{\SL}(\C,V_{\vlambda}\otimes V_{\nu^*})$. The goal of this section is to prove the following result:

\begin{theorem} Let $\vlambda$ be a sequence of dominant minuscule weights. Relative to the usual ordering on $P_\vlambda(\vnu)$, there is an upper unitriangular change of basis between $\Psi_{\nu}(T_\vmu)$ for $\vmu\in P_\vlambda(\vnu)$ and the Satake basis of $\Hom_{\SL}(\C,V_\vlambda\otimes V_{\nu^*})$. In particular, the set of webs arising from the triangular diagrams $T_\vmu$ form a basis.
\label{th:subrep}\end{theorem}

This result follows directly from the next lemma:

\begin{lemma} $$\Psi_{\nu}(T_\vmu)=[\overline{Q(A(\vlambda\sqcup\nu^*,\vmu))}]+\sum_{\stackrel{\vgamma\in P_\vlambda(\nu)}{\vgamma<\vnu}}c_{\vmu}^{\vgamma}[\overline{Q(A(\vlambda\sqcup\nu^*,\vgamma))}].$$
\label{lem:keylemma}\end{lemma}

To prove this result requires the use of the geometric Satake correspondence of Lusztig \cite{Lusztig:qanalog}, Ginzburg \cite{Ginzburg:loop}, and Mirkovi\'c-Vilonen \cite{MV:geometric}:

\begin{theorem} The representation category $\rep(G)$ is equivalent as a pivotal category to the category $\perv(\Gr)$ of perverse sheaves on the affine Grassmannian $\Gr$ constructible with respect to the stratification by the orbits $\Gr_\lambda$. 
\label{th:satake} \end{theorem}

Using this theorem we can transport the problem to the category $\perv(\Gr)$. Recall that the orbits $\Gr_\lambda$ are smooth and simply connected of complex dimension $\braket{\lambda,\rho}$ where $\rho$ is half the sum of the positive roots of $\SL$. For any stratified algebraic variety $M$ there is a perverse sheaf $\IC_M$, the intersection cohomology sheaf which restricts to the shifted constant sheaf $\C[\dim_\C M]$ on the open stratum. In the case that $M$ is smooth, $\IC_M=\C_M[\dim_\C M]=\D_M[-\dim_C M]$ where $\C_M$ is the locally constant sheaf on $M$ and $\D_M$ is the dualizing sheaf on $M$. Then for each orbit $\Gr_\lambda$, there is one simple object in $\perv(\Gr)$: $\IC_{\overline{\Gr_\lambda}}$ which we extend by zero to all of $\Gr$. Under the geometric Satake correspondence, $V_\lambda$ is sent to $\IC_{\overline{\Gr_\lambda}}$ and if $\vlambda$ is a sequence of dominant weights, $V_\vlambda$ corresponds to $(m_{\vlambda})_*\IC_{\overline{\Gr_{\vlambda}}}$ where $m_{\vlambda}:\overline{\Gr_{\vlambda}}\rightarrow\Gr$ is the multiplication morphism. By general principles we then have \begin{multline*}\Hom_{\SL}(\C,V_\vlambda)\cong\Hom_{\perv(\Gr)}(\C_{\Gr_0},(m_{\vlambda})_*\IC_{\overline{\Gr_{\vlambda}}})\stackrel{\iota}{\cong} \\ \Hom_{\perv(\Gr)}(\C_{\Gr_0},(m_{\vlambda})_*\D_{\overline{\Gr_{\vlambda}}}[-\dim_\C \overline{\Gr_{\vlambda}}])\cong \\ H_{\top}(\overline{\Gr_{\vlambda}})=H_{\top}(F(\vlambda)).\end{multline*} Here $\iota$ is canonical map between $\IC_{\overline{\Gr_{\vlambda}}}$ and $*\D_{\overline{\Gr_{\vlambda}}}[-\dim_\C \overline{\Gr_{\vlambda}}]$.

%In the case where $\lambda$ or $\vlambda$ are minuscule then $\overline{\Gr_\lambda}=\Gr_\lambda$ and $\Gr_\vlambda$ are smooth. In this case \cite[Thm. 8.6.7]{CG:complex} tells us that the composition of morphisms in the subcategory of $\perv(\Gr)$ generated $\IC_{\Gr_\lambda}$ for $\lambda$ minuscule is equivalent convolution in homology. Thus if $\Gr_{\vlambda\sqcup\nu^*}$ was smooth, then we could prove the theorem by convolving $\pi_*c(T_\vmu)$ with the class in $H_{\top}(\Gr_{\vlambda\sqcup\vnu^*}\times_\Gr\Gr_{\vlambda\sqcup\nu^*})$ corresponding to $\tilde{\phi}=id_{V_{\vlambda}}\otimes\phi^*$. Unfortunately since $\nu$ is not necessarily minuscule, we must proceed differently.

\begin{proof}[Proof of Lemma~\ref{lem:keylemma}]
Under the geometric Satake correspondence, $\tilde{\phi}$ corresponds to a map $$\tilde{\phi}:\Hom(\C_{\Gr_0},(m_{\vlambda\sqcup\vnu^*})_*\D_{\overline{\Gr_{\vlambda\sqcup\vnu^*}}}[-d])\rightarrow \Hom(\C_{\Gr_0},(m_{\vlambda\sqcup\nu^*})_*\IC_{\overline{\Gr_{\vlambda\sqcup\nu^*}}}).$$ Here $$d=\braket{|\vlambda\sqcup\vnu^*|,\rho}=\braket{\lambda_1+\cdots+\lambda_k+\nu_1^*+\cdots+\nu_l^*,\rho}=\dim_\C \overline{\Gr_{\vlambda\sqcup\vnu^*}}.$$ On the other hand, we consider the map $\pi:\overline{\Gr_{\vlambda\sqcup\vnu*}}\rightarrow\overline{\Gr_{\vlambda\sqcup\nu^*}}$, the projection onto the first $k$ and the last factors. Then $\pi$ restricts to the map $$t:F(\vlambda\sqcup\vnu^*)\rightarrow F(\vlambda\sqcup\nu^*),$$ which projects onto the first $k$ factors. This in turn defines a pushforward in homology $$t_*:H_{2d}(F(\vlambda\sqcup\vnu^*))\rightarrow H_{2d}(F(\vlambda\sqcup\nu^*)).$$ $F(\vlambda\sqcup\vnu^*)$ has pure dimension $d$ and its components are $\overline{Q(A(\vlambda\sqcup\vnu^*,\vmu))}$ for $\vmu\in P_{\vlambda\sqcup\vnu^*}(0).$ The variety $F(\vlambda\sqcup\nu^*)$ is no longer pure dimensional, but its irreducible components of dimension $d$ are $\overline{Q(A(\vlambda\sqcup\nu^*,\vmu))}$ where $\vmu\in P_\vlambda(\nu).$ If $\vmu\in P_{\vlambda\sqcup\vnu^*}(0)$ then let $\vmu_t$ be its truncation at position $k$. Then $$t(Q(A(\vlambda\sqcup\vnu^*,\vmu)))=Q(A(\vlambda\sqcup\mu_k^*,\vmu_t)).$$ It follows that $t_*$ sends $[\overline{Q(A(\vlambda\sqcup\vnu^*,\vmu))}]$ to $0$ when $\mu_k<\nu$ since the dimension of $\overline{Q(A(\vlambda\sqcup\mu_k^*,\vmu_t))}$ is $\braket{|\vlambda\sqcup\mu_k^*|,\rho}<d$. Otherwise $t$ is an isomorphism from $Q(A(\vlambda\sqcup\vnu^*,\vmu))$ to its image $Q(A(\vlambda\sqcup\nu^*,\vmu_t))$, so $t$ has degree $1$ here and it sends $[\overline{Q(A(\vlambda\sqcup\vnu^*,\vmu))}]$ to $[\overline{Q(A(\vlambda\sqcup\nu^*,\vmu_t))}]$.

Thus if $\tilde{\phi}=t_*$ on homology we will be done, since
$$\begin{aligned}
\Psi_\nu(T_\vmu) = & t_*(\Psi(T_\vmu)) \\
= & t_*([\overline{Q(A(\vlambda\sqcup\vnu^*,\tilde{\vmu}))}])+\sum_{\stackrel{\vgamma\in P_{\vlambda\sqcup\vnu^*}(0)}{\vgamma<\tilde{\vmu}}}c_{\tilde{\vmu}}^{\vgamma}t_*([\overline{Q(A(\vlambda\sqcup\vnu^*,\vgamma))}]) \\
= & [\overline{Q(A(\vlambda\sqcup\nu^*,\vmu))}])+\sum_{\stackrel{\vgamma\in P_{\vlambda}(\nu)}{\vgamma<\vmu}}c_{\tilde{\vmu}}^{\vgamma}[\overline{Q(A(\vlambda\sqcup\nu^*,\vgamma))}]
\end{aligned}$$
Note that $\vgamma\in P_{\vlambda}(\nu)$ uniquely extends to $\vgamma\in P_{\vlambda\sqcup\vnu^*}(0)$ so the coefficient $c_{\tilde{\vmu}}^{\vgamma}$ in the last line is the same as in the second.

Now recall that $\tilde{\phi}$ is composition with the map $$id_{(m_{\vlambda})_*\D_{\overline{\Gr_{\vlambda}}}[-\braket{|\vlambda|,\rho}]}\otimes\phi\in\Hom((m_{\vlambda\sqcup\vnu^*})_*\D_{\overline{\Gr_{\vlambda\sqcup\vnu^*}}}[-d],(m_{\vlambda\sqcup\nu^*})_*\IC_{\overline{\Gr_{\vlambda\sqcup\nu^*}}}),$$ where $\phi$ is the canonical map in $$\Hom((m_{\vnu^*})_*\D_{\overline{\Gr_{\vnu^*}}}[-\braket{|\vnu^*|,\rho}],\IC_{\overline{\Gr_{\nu^*}}}).$$ Since $m_{\vlambda\sqcup\vnu^*}=m_{\vlambda\sqcup\nu^*}\pi$ we have $(m_{\vlambda\sqcup\vnu^*})_*=(m_{\vlambda\sqcup\nu^*})_*\pi_*$, which gives the following claim:

\begin{lemma} There exists a map $\psi\in\Hom(\pi_*\D_{\overline{\Gr_{\vlambda\sqcup\vnu^*}}}[-d],\IC_{\overline{\Gr_{\vlambda\sqcup\nu^*}}})$ such that $(m_{\vlambda\sqcup\nu^*})_*\psi=id_{(m_{\vlambda})_*\D_{\overline{\Gr_{\vlambda}}}[-\braket{|\vlambda|,\rho}]}\otimes\phi$.
\end{lemma}

This map exists because by the definition of the convolution tensor product we have $$\Hom((m_{\vnu^*})_*\D_{\overline{\Gr_{\vnu^*}}}[-\braket{|\vnu^*|,\rho}],\IC_{\overline{\Gr_{\nu^*}}})\cong\Hom(\pi_*\D_{\overline{\Gr_{\vlambda\sqcup\vnu^*}}}[-d],\IC_{\overline{\Gr_{\vlambda\sqcup\nu^*}}})$$ and the image of $\phi$ under this map is $\psi$. In fact, this map is characterized by the fact that it is the identity when restricted to the open dense stratum of $\vlambda\sqcup\nu^*$. Thus we have the following diagram of complexes of constructible sheaves on $\overline{\Gr_{\vlambda\sqcup\nu^*}}$:
$$\begin{tikzpicture}
\matrix (m) [matrix of math nodes, row sep=3em, column sep=3em, text height=1.5ex, text depth=0.25ex]
{ \pi_!\D_{\overline{\Gr_{\vlambda\sqcup\vnu^*}}}[-d]=\pi_*\D_{\overline{\Gr_{\vlambda\sqcup\vnu^*}}}[-d] & \IC_{\overline{\Gr_{\vlambda\sqcup\nu^*}}} \\
 & \D_{\overline{\Gr_{\vlambda\sqcup\nu^*}}}[-d] \\ };
\path[->] (m-1-1) edge node[auto] {$\psi$} (m-1-2) edge node[auto] {$\tilde{\pi}$} (m-2-2) (m-1-2) edge node[auto] {$\iota$} (m-2-2);
\end{tikzpicture}$$

Once again $\iota$ is the canonical map between the $\IC$ sheaf of a variety and the shifted dualizing sheaf and $\tilde{\pi}$ is the map obtained from applying the adjunction $\pi_!\vdash \pi^!$ to the identity morphism in $\Hom(\D_{\overline{\Gr_{\vlambda\sqcup\vnu^*}}}[-d],\D_{\overline{\Gr_{\vlambda\sqcup\vnu^*}}}[-d])$ after realizing that $\D_{\overline{\Gr_{\vlambda\sqcup\vnu^*}}}[-d]=\pi^!\D_{\overline{\Gr_{\vlambda\sqcup\nu^*}}}[-d]$.

\begin{lemma} The above diagram commutes.
\end{lemma}

\begin{proof} First we notice that since $$\Hom(\D_{\overline{\Gr_{\vlambda\sqcup\vnu^*}}}[-d],\D_{\overline{\Gr_{\vlambda\sqcup\vnu^*}}}[-d])\cong\Hom(\pi_!\D_{\overline{\Gr_{\vlambda\sqcup\vnu^*}}}[-d],\D_{\overline{\Gr_{\vlambda\sqcup\nu^*}}}[-d])$$ via the adjunction. Thus since $\Hom(\D_{\overline{\Gr_{\vlambda\sqcup\vnu^*}}}[-d],\D_{\overline{\Gr_{\vlambda\sqcup\vnu^*}}}[-d])$ is one dimensional, so is $\Hom(\pi_!\D_{\overline{\Gr_{\vlambda\sqcup\vnu^*}}}[-d],\D_{\overline{\Gr_{\vlambda\sqcup\nu^*}}}[-d])$. Thus the above diagram must commute up to some constant $c$. By restricting to the open dense stratum $\Gr_{\vlambda\sqcup\nu^*}\subset\overline{\Gr_{\vlambda\sqcup\nu^*}}$, we can discover what this constant is. After restricting, each complex becomes the constant sheaf $\C_{\Gr_{\vlambda\sqcup\nu^*}}[d]$ and the maps $\tilde{\pi}$ and $\iota$ both become the identity map. As stated before, the map $\Psi$ is also the identity map. Thus the constant is one and the diagram commutes.
\end{proof}

Thus the diagram still commutes after applying the functors $(m_{\vlambda\sqcup\nu^*})_*$ and $\Hom(\C_{\Gr_0},\cdot)$, i.e. we have the following commutative diagram:
$$\begin{tikzpicture}
\matrix (m) [matrix of math nodes, row sep=3em, column sep=3em, text height=1.5ex, text depth=0.25ex]
{ \Hom(\C_{\Gr_0},(m_{\vlambda\sqcup\vnu^*})_*\D_{\overline{\Gr_{\vlambda\sqcup\vnu^*}}}[-d]) & \Hom(\C_{\Gr_0},(m_{\vlambda\sqcup\nu^*})_*\IC_{\overline{\Gr_{\vlambda\sqcup\nu^*}}}) \\
 & \Hom(\C_{\Gr_0},(m_{\vlambda\sqcup\nu^*})_*\D_{\overline{\Gr_{\vlambda\sqcup\nu^*}}}[-d]) \\ };
\path[->] (m-1-1) edge node[auto] {$\tilde{\phi}$} (m-1-2) edge node[auto] {$\tilde{\pi}$} (m-2-2) (m-1-2) edge node[below,sloped,inner sep=0.5pt] {$ \sim $} node[auto] {$\iota$} (m-2-2);
\end{tikzpicture}$$

The map $\iota$ becomes an isomorphism and we have $\tilde{\phi}=\tilde{\pi}$. We have $$\Hom(\C_{\Gr_0},(m_{\vlambda\sqcup\vnu^*})_*\D_{\overline{\Gr_{\vlambda\sqcup\vnu^*}}}[-d])\cong H_\top(F(\vlambda\sqcup\vnu^*))$$ and $$\Hom(\C_{\Gr_0},(m_{\vlambda\sqcup\nu^*})_*\D_{\overline{\Gr_{\vlambda\sqcup\nu^*}}}[-d])\cong H_\top(F(\vlambda\sqcup\nu^*)),$$ and under these isomorphisms $\tilde{\pi}$ becomes $t_*$. Thus $\tilde{\phi}$ is $t_*$ on homology as needed.
\end{proof}

\section{Agreement with Earlier Work}
\subsection{$\SL[2]$}
In the case of $\SL[2]$, due to Frenkel and Khovanov \cite{FK:canonical}, the set of cup diagrams of length $2n$ is in bijection with Lusztig's dual canonical basis for $\Hom_{\SL[2]}(V_0,V_1^{\otimes 2n})$. The set of cup diagram or crossless planar matchings is the ways of connecting an even number of points on a line via edges with no crossings.

Such a planar matching is associated with a $2$ by $n$ Young Tableaux and hence with a minuscule Littelmann path. If position $i$ is the start of an edge, $i$ is in the top row of the Young Tableaux and the $i$-th weight in the path is $(l_i,0)$ where $l_i$ is the number of arcs that start at before position $i$ and end after.

We must check that from the given minuscule Littelmann path we recreate the same cup diagram. Since $\SL[2]$ has only one minuscule weight $\omega_1=(1,0)$, it follows $W\omega_1=\{(1,0),(0,1)\}$ and thus we have two possible irreducible triangles of length $1$:
$$
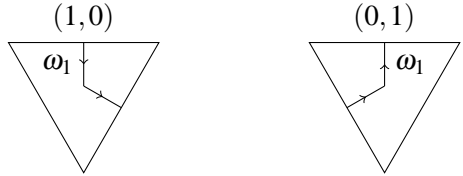
\begin{tikzpicture}[ar/.style={postaction={decorate,decoration={markings,mark=at position .5 with {\arrow{>}}}}}]
\draw (0,0) -- +(60:2) -- ++(120:2) node[midway,above] {$(1,0)$} -- cycle;
\draw[ar] ($ (60:1) + (120:1) $) -- ($ (60:0.666) + (120:0.666) $) node[midway,left] {$\omega_1$};
\draw[ar] ($ (60:0.666) + (120:0.666) $) -- (60:1);
\draw (4,0) -- +(60:2) -- ++(120:2) node[midway,above] {$(0,1)$} -- cycle;
\draw[ar] ($ (4,0) + (60:0.666) + (120:0.666) $) -- ($ (4,0) + (60:1) + (120:1) $) node[midway,right] {$\omega_1$};
\draw[ar] ($ (4,0) + (120:1) $) -- ($ (4,0) + (60:0.666) + (120:0.666) $);
\end{tikzpicture}$$

Thus the three possible diamonds are

$$
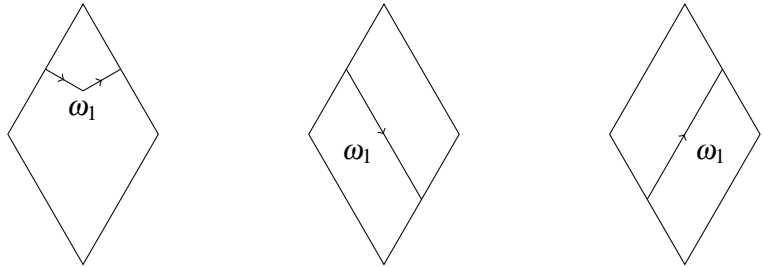
\begin{tikzpicture}[ar/.style={postaction={decorate,decoration={markings,mark=at position .5 with {\arrow{>}}}}}]
\draw (0,0) -- ++(60:2) -- ++(120:2) -- ++(240:2) -- cycle;
\draw[ar] ($ (120:2) + (60:1) $) -- ($ (120:1.333) + (60:1.333) $) node[below] {$\omega_1$};
\draw[ar] ($ (120:1.333) + (60:1.333) $) -- ($ (120:1) + (60:2) $);
\draw (4,0) -- ++(60:2) -- ++(120:2) -- ++(240:2) -- cycle;
\draw[ar] ($ (4,0) + (120:2) + (60:1) $) -- ($ (4,0) + (60:1) $) node[midway,below left] {$\omega_1$};
\draw (8,0) -- ++(60:2) -- ++(120:2) -- ++(240:2) -- cycle;
\draw[ar] ($ (8,0) + (120:1) $) -- ($ (8,0) + (60:2) + (120:1) $) node[midway,below right] {$\omega_1$};
\end{tikzpicture}$$

Here we have not shown the edges that are labeled zero by the construction. Since the dual weight to $\omega_1$ is $\omega_1$ the orientations that appear in the construction do not matter, thus we can drop the orientation and labels all together. Since the diagram corresponding to minuscule Littelmann path had no edges incident to the left or right sides of the triangle and there are no crossings in the diamonds, we get a crossless planar matching. This edges in this matching start at the triangles labelled $(1,0)$ and end at those labelled $(0,1)$, so each minuscule Littelmann path gives a unique crossless planar matching. Since the number of minuscule Littelmann paths is $C_n$, which is also the number of crossless planar matchings, the two constructions agree.
\subsection{$\SL[3]$}

For $\SL[3]$, the work of Greg Kuperberg gives a set of webs, i.e the graphs generated by two vertices $$
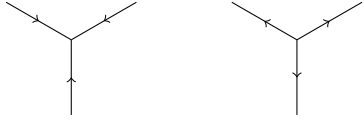
\begin{tikzpicture}[ar/.style={postaction={decorate,decoration={markings,mark=at position .5 with {\arrow{>}}}}}] \draw[ar] (3,0) -- ++(30:1);\draw[ar] (3,0) -- ++(150:1);\draw[ar] (3,0) -- ++(270:1);\draw[ar] (30:1) -- (0,0);\draw[ar] (150:1) -- (0,0);\draw[ar] (270:1) -- (0,0);\end{tikzpicture}$$

The set of basis webs are those graphs with no faces of degree 2 or 4. Since $\omega_1=\omega_2^*$, we have oriented all edges so that they are labelled with the weight $\omega_1.$ Each web has coherent geodesics and a path of minimal weight in the dual diskoid is also one of minimal length. Each graph corresponds to a unique minuscule Littelmann path and given a minuscule Littelmann path, there is an algorithm in \cite{Kuperberg:notdual} which generates the basis web from the given path.

\begin{lemma} Given a minuscule Littelmann path $\mu_i$, $T_{\vmu}$ is a non-elliptic web, i.e it has no faces of degree $4$ or lower.
\end{lemma}

\begin{proof} Consider an internal face $f$ and consider the triangle of length one, or diamond closest to the top of $T_\vmu$ that also contains part of the selected face. Since the face is internal there are two possibilities: 
$$
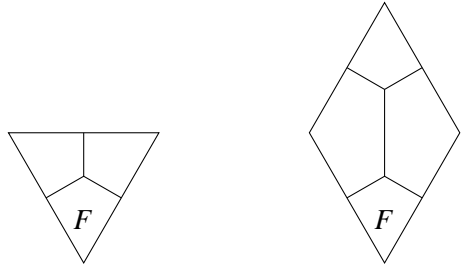
\begin{tikzpicture}
\draw (0,0) -- (60:2) -- (120:2) -- cycle;
\draw (60:1) -- ($ (60:0.666) + (120:0.666) $) -- (120:1);
\draw ($ (60:0.666) + (120:0.666) $) -- ($ (60:1) + (120:1) $);
\draw ($ (60:0.333) + (120:0.333) $) node {$F$};
\draw (4,0) -- ++(60:2) -- ++(120:2) -- ++(240:2) -- cycle;
\draw ($ (4,0) + (60:1) $) -- ($ (4,0)+(60:0.666) + (120:0.666) $) -- ($ (4,0)+(120:1) $);
\draw ($ (4,0) + (60:2) + (120:1) $) -- ($ (4,0)+(60:1.333) + (120:1.333) $) -- ($ (4,0)+(120:2)+(60:1) $);
\draw ($ (4,0)+(60:1.333) + (120:1.333) $) -- ($ (4,0)+(60:0.666) + (120:0.666) $);
\draw ($ (4,0)+(60:0.333) + (120:0.333) $) node {$F$};
\end{tikzpicture}$$
Either of these possibilities would add one vertex to the face. Considering a diamond that contains a left or rightmost point, it follows that the diamonds would be as above and thus each add 2 more vertices to the face. Thus an internal face must have at least $5$ vertices. Then $T_\vmu$ is an $\SL[3]$ web with faces no faces of degree $4$ or $2$ and thus it is non-elliptic.
\end{proof}

Thus given a minuscule sequence $\vlambda$, the set of webs $\{T_\vmu|\vmu\in P_\vlambda\}$ is contained in the set of basis webs and is the same size since they are both index by $P_\vlambda$, thus every basis web is generated. For each minuscule Littelmann path $\vmu$ the growth algorithm of Kuperberg gives the gives the same web as $T_\vmu$.

\subsection{$\SL$ for weights $\omega_1$ and $\omega_{n-1}$}
In \cite{Westbury:general} Westbury produces a set of basis webs for $\Hom_{\SL}(\C,V_\vlambda)$ when $\vlambda$ is a sequence of weights $\omega_1$, and $\omega_{n-1}$. In this case, the minimal diagrams of length $1$ are simply:
$$
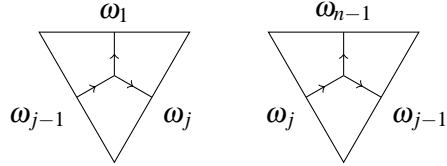
\begin{tikzpicture}[ar/.style={postaction={decorate,decoration={markings,mark=at position .5 with {\arrow{>}}}}}]
\draw (0,0) -- (60:2) -- (120:2) -- cycle;
\draw[ar] ($ (60:0.666) + (120:0.666) $) -- (60:1) node[below right] {$\omega_j$};
\draw[ar] (120:1) node[below left] {$\omega_{j-1}$} -- ($ (60:0.666) + (120:0.666) $);
\draw[ar] ($ (60:0.666) + (120:0.666) $) -- ($ (60:1) + (120:1) $) node[above] {$\omega_1$};
\end{tikzpicture}
\hskip20pt
\begin{tikzpicture}[ar/.style={postaction={decorate,decoration={markings,mark=at position .5 with {\arrow{>}}}}}]
\draw (0,0) -- (60:2) -- (120:2) -- cycle;
\draw[ar] ($ (60:0.666) + (120:0.666) $) -- (60:1) node[below right] {$\omega_{j-1}$};
\draw[ar] (120:1) node[below left] {$\omega_j$} -- ($ (60:0.666) + (120:0.666) $);
\draw[ar] ($ (60:0.666) + (120:0.666) $) -- ($ (60:1) + (120:1) $) node[above] {$\omega_{n-1}$};
\end{tikzpicture}
$$

In this case the basis web already come from the triangular diagrams generated by these length one diagrams. Since there is is only one diagram of length one for each element in the orbits $W\omega_1$ and $W\omega_{n-1}$, there is a unique triangular diagram for each minuscule path. As in the $\SL[3]$ and $\SL[2]$ cases, the resultant basis in $\Hom_{\SL}(\C,V_\vlambda)$ is invariant under rotation.

\section{Specialization to $\SL[4]$}\label{sec:sl4}
To simplify the notation in this section, an unoriented double edge in a triangular diagram is considered to have label $\omega_2$ and any directed edge is considered to have label $\omega_1$. This is equivalent to the previous description of webs since $\omega_2$ is self dual, so the edges labelled $\omega_2$ have no inherent orientation, and $\omega_1$ is dual to $\omega_3$, so all edges with label $\omega_3$ can be reversed to have label $\omega_1$. As mentioned earlier, for $\SL[2]$, $\SL[3]$ there was a single irreducible triangular diagram of length $1$ for each weight in the the Weyl orbit. Under the construction of webs via triangular diagrams this means that there was only one possibility for a set of basis webs. On the other hand when we include $\omega_2$ in $\SL[4]$, we have the following choices:
$$
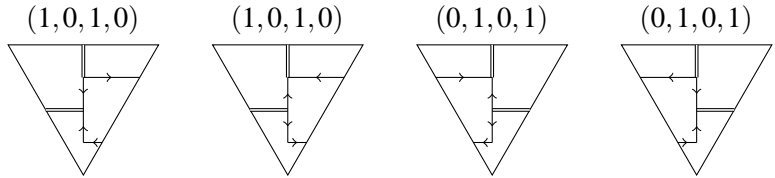
\begin{tikzpicture}[cm={0.5, 0.866, -0.5,0.866,(0,0)},ar/.style={postaction={decorate,decoration={markings,mark=at position .5 with {\arrow{>}}}}}]
\draw (0,0) -- (0,2) -- (2,0) node[midway,above] {$(1,0,1,0)$} -- cycle;
\draw[double] (1,1) -- (0.75,0.75);
\draw[double] (0,1) -- (0.5,0.5);
\draw[ar] (0.75,0.75) -- (0.5,0.5);
\draw[ar] (0.25,0.25) -- (0.5,0.5);
\draw[ar] (0.75,0.75) -- (1.5,0);
\draw[ar] (0.5,0) -- (0.25,0.25);
\end{tikzpicture}\hskip20pt
\begin{tikzpicture}[cm={0.5, 0.866, -0.5,0.866,(0,0)},ar/.style={postaction={decorate,decoration={markings,mark=at position .5 with {\arrow{>}}}}}]
\draw (0,0) -- (0,2) -- (2,0) node[midway,above] {$(1,0,1,0)$} -- cycle;
\draw[double] (1,1) -- (0.75,0.75);
\draw[double] (0,1) -- (0.5,0.5);
\draw[ar] (0.5,0.5) -- (0.75,0.75);
\draw[ar] (0.5,0.5) -- (0.25,0.25);
\draw[ar] (1.5,0) -- (0.75,0.75);
\draw[ar] (0.25,0.25) -- (0.5,0);
\end{tikzpicture}\hskip20pt
\begin{tikzpicture}[cm={0.5, 0.866, -0.5,0.866,(0,0)},ar/.style={postaction={decorate,decoration={markings,mark=at position .5 with {\arrow{>}}}}}]
\draw (0,0) -- (0,2) -- (2,0) node[midway,above] {$(0,1,0,1)$} -- cycle;
\draw[double] (1,1) -- (0.75,0.75);
\draw[double] (1,0) -- (0.5,0.5);
\draw[ar] (0.5,0.5) -- (0.75,0.75);
\draw[ar] (0.5,0.5) -- (0.25,0.25);
\draw[ar] (0,1.5) -- (0.75,0.75);
\draw[ar] (0.25,0.25) -- (0,0.5);
\end{tikzpicture}\hskip20pt
\begin{tikzpicture}[cm={0.5, 0.866, -0.5,0.866,(0,0)},ar/.style={postaction={decorate,decoration={markings,mark=at position .5 with {\arrow{>}}}}}]
\draw (0,0) -- (0,2) -- (2,0) node[midway,above] {$(0,1,0,1)$} -- cycle;
\draw[double] (1,1) -- (0.75,0.75);
\draw[double] (1,0) -- (0.5,0.5);
\draw[ar] (0.75,0.75) -- (0.5,0.5);
\draw[ar] (0.25,0.25) -- (0.5,0.5);
\draw[ar] (0.75,0.75) -- (0,1.5);
\draw[ar] (0,0.5) -- (0.25,0.25);
\end{tikzpicture}$$
When forming a triangular diagram using $T_{(1,0,1,0)}$ or $T_{(0,1,0,1)}$ it is reasonable to pick either of these diagrams. In most cases there is a unique choice for the diagrams $T_{(1,0,1,0)}$ and $T_{(0,1,0,1)}$ that produces a triangular diagram with the least number of vertices. When the choice is not unique, we will show that all choices produce the same web vector. This happens only when we have a minuscule Littelmann path $\vlambda$ such that there exists indices $i$ and $j$ with the property that $\lambda_{i+1}-\lambda_i=(1,0,1,0)$, $\lambda_{j+1}-\lambda_j=(0,1,0,1)$ and the triangular diagram corresponding to the sub path $\lambda_{i+1},\cdots,\lambda_{j-1}$ has only edges of weight $\omega_2$ appearing along its sides. For each such pair $(i,j)$ we will show that there are two choices that result in a minimal number of vertices.

Call a path $\vmu$ $\gamma$-dominant if $\mu_i+\gamma$ is dominant for all $i$. A sequence of minuscule weights $\vnu$ is $\gamma$ dominant if its sequence of partial sums is $\gamma$ dominant. 

\begin{lemma} If $\vnu$ is a sequence of minuscule weights, the multiset of edges weights on the left and and right edges of $\bigotimes_i T_{\nu_i}$ is independent of the choice of $T_{\nu_i}$ when $\nu_i=(1,0,1,0)$ or $(0,1,0,1)$.
\end{lemma}

\begin{proof}
Since the multiset of weights on the left and right sides of the irreducible triangles of length $1$ does not depend on this choice and the diagram, let $L_i$ and $R_i$ be the multiset of weights on the left and right of $T_{\nu_i}$. Then the rule for the product of diagrams tells us that the left and right side of $\bigotimes_i T_{\nu_i}$ have multisets of weights $\cup L_i \setminus \cup R_i$ and $\cup R_i \setminus \cup L_i$ respectively.
\end{proof}

\begin{corollary} 
Let $\vnu$ be a $\gamma$-dominant minuscule path. The multiset of weights on the left of $T_\vnu$ is contained in the multiset of fundamental weights that sum to $\gamma$.
\end{corollary}

\begin{theorem}
Let $\vlambda$ be a sequence of dominant minuscule weights and $\vmu\in P_\vlambda$. Let $\vnu$ its sequence of successive differences of $\vmu$. There exists a sequence $T_i$ of minuscule diagrams of length $1$ for $\nu_i$ such that the web corresponding to $\bigotimes T_i$ has the least number of vertices over all choices for $T_i$. For each pair $(i,j)$ such that $\nu_i=(1,0,1,0)$, $\nu_{j+1}=(0,1,0,1)$ along with $\nu_{i+1},\cdots,\nu_j$ is $k\omega_2$ dominant and $\sum_{l=i+1}^j\nu_l = k'\omega_2$ for some $k$, $k'$ there are two choices for the diagrams $T_i$ and $T_{j+1}$.
\end{theorem}

\begin{proof}
For any $\nu_i\not\in\{(1,0,1,0),(0,1,0,1)\}$, we only have one choice of irreducible diagram of length one, set $T_i=T_{\nu_i}$. For $\nu_i=(1,0,1,0)$, let $j$ be maximal such that $\nu_{i+1},\cdot,\nu_j$ is $k\omega_2$ dominant for some $k$. Thus by the above corollary $\bigotimes{k=i+1}^j T_k$ has only unoriented double edges on the left. Since $j$ is maximal, this implies that there is at least one oriented edge on the left of $\bigotimes_{l=i+1}^{j+1} T_{\mu_l}$. There can be at most $2$, since no irreducible diagram of length $1$ has more than $2$ edges on the left. Suppose there is $1$ oriented edge on the left, if it is oriented from left to right we choose $T_i$ to be the first of the two diagram for $\nu_i$, otherwise we chose the second. The previous lemma implies that this edge is independent of any choices for $T_k$ for $i<k\leq j$. In the case that two oriented edges are present we must have $\nu_{j+1}=(0,1,0,1)$ since any other weight would add $0$ or $1$ edges to the left. In this case there are two choices which minimize crossings: both the first or both the second choices for $T_i$ and $T_{j+1}$. The argument is similar for $\nu_i=(0,1,0,1)$.

In order to see that these choices result in the least number of vertices, making the the other choice at any stage results in a diagram with a strictly greater number of vertices. By the previous lemma, these choices do not interfere with each other, so each time we switch the choice in the above construction, we increase the number of vertices.
\end{proof}

\begin{lemma} The invariant vector assigned to the preceding web is independent of the choices made during the construction of the triangular diagram.
\end{lemma}

Figure~\ref{sl4choice} shows one possible choice for the strip of diamonds that join the two triangles
\begin{figure*} \begin{tikzpicture}[ar/.style={postaction={decorate,decoration={markings,mark=at position .5 with {\arrow{>}}}}},cm={0.5, 0.866, -0.5,0.866,(0,0)},scale=0.5]
\draw (0,0) -- (12,0) -- (9,3) -- (3,3) -- (3,9) -- (0,12) -- cycle;
\draw (3,0) -- (3,3) (0,3) -- (3,3) (6,0) -- (6,3) (0,6) -- (3,6) (9,0) -- (9,3) (0,9) -- (3,9);

\draw[double] (0,7.019) -- (0.680,7.019); \draw[ar] (1.160,7.5) -- (0.680,7.019); \draw[double] (1.160,7.5) -- (1.841,7.5); \draw[ar] (1.841,7.5) -- (2.321,7.981); \draw[double] (2.321,7.981) -- (3,7.981);
\draw[ar] (1.841,7.5) -- (1.841,3);
\draw[ar] (0.680,3) -- (0.680,7.019);
\draw[ar] (1,10) -- (2.321,7.981);
\draw[double] (1,10) -- (1.5,10.5); \draw[ar] (1,10) -- (0.5,9.5); \draw[ar] (1.160,7.5) -- (0.5,9.5);
\draw[double] (0,9.5) -- (0.5,9.5);

\draw[double] (7.019,0.680) -- (7.019,0); \draw[ar] (7.019,0.680) -- (7.5,1.160); \draw[double] (7.5,1.841) -- (7.5,1.160); \draw[ar] (7.981,2.321) -- (7.5,1.841); \draw[double] (7.981,3) -- (7.981,2.321);
\draw[ar] (3,1.841) -- (7.5,1.841);
\draw[ar] (7.019,0.680) -- (3,0.680);
\draw[ar] (7.981,2.321) -- (10,1);
\draw[double] (10.5,1.5) -- (10,1); \draw[ar] (9.5,0.5) -- (10,1); \draw[ar] (9.5,0.5) -- (7.5,1.160);
\draw[double] (9.5,0.5) -- (9.5,0);

\draw[ar] (3,0.680) .. controls (2,0.680) and (0.680,2) .. (0.680,3);
\draw[ar] (1.841,3) .. controls (1.841,2.5) and (2.5,1.841) .. (3,1.841);
\draw (5,5) node {$A$};
\end{tikzpicture}\caption{A possible choice in the construction of a triangular diagram.}\label{sl4choice} \end{figure*}
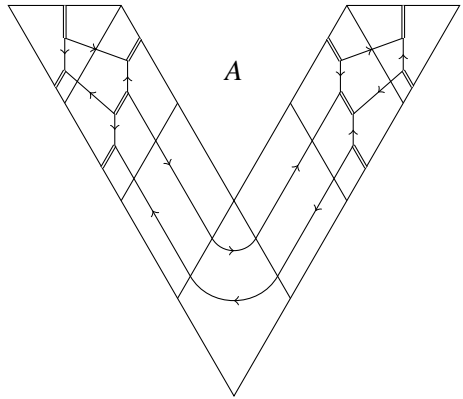
By supposition the triangular diagram taking the place of $A$ in Figure~\ref{sl4choice} has only edges of weight $\omega_2$ incident on the left and right sides. Thus, in general, the only possible diamonds that appear in the strip are those pictured in Figure~\ref{sl4choice}. In \cite{Kim:thesis} Kim proposes a set of generators of the kernel of the map $\Psi$ from linear combinations of webs to invariant vectors. Since these relations at least generation a subspace of the kernel, we can use them to see that two diagrams result in the same invariant vector. Of the relations that Kim defines, the following $3$ will be needed to prove the lemma:
\begin{equation}\label{sl4ker1}\begin{split}\begin{tikzpicture}[ar/.style={postaction={decorate,decoration={markings,mark=at position .5 with {\arrow{>}}}}},scale=0.75]
\draw[double] (0.5,0.5) -- (1,1) (2.5,0.5) -- (2,1) (2.5,2.5) -- (2,2) (0.5,2.5) -- (1,2);
\draw[ar] (1,1) -- (1,2);
\draw[ar] (1,1) -- (2,1);
\draw[ar] (2,2) -- (1,2);
\draw[ar] (2,2) -- (2,1);
\end{tikzpicture}\end{split}\hskip5pt=\hskip5pt\begin{split}
\begin{tikzpicture}[ar/.style={postaction={decorate,decoration={markings,mark=at position .5 with {\arrow{>}}}}},scale=0.75]
\draw[double] (0.5,0.5) -- (1,1) (2.5,0.5) -- (2,1) (2.5,2.5) -- (2,2) (0.5,2.5) -- (1,2);
\draw[ar] (1,2) -- (1,1);
\draw[ar] (2,1) -- (1,1);
\draw[ar] (2,1) -- (2,2);
\draw[ar] (1,2) -- (2,2);
\end{tikzpicture}\end{split}
\end{equation}
\begin{equation}\label{sl4ker2}\begin{split}\begin{tikzpicture}[ar/.style={postaction={decorate,decoration={markings,mark=at position .5 with {\arrow{>}}}}},scale=0.75]
\draw[double] (0.5,0.5) -- (0.5,1.5);
\draw[ar] (0,0) -- (0.5,0.5);
\draw[ar] (1,0) -- (0.5,0.5);
\draw[ar] (0,2) -- (0.5,1.5);
\draw[ar] (1,2) -- (0.5,1.5);
\end{tikzpicture}\end{split} =\hskip5pt \begin{split}
\begin{tikzpicture}[ar/.style={postaction={decorate,decoration={markings,mark=at position .5 with {\arrow{>}}}}},scale=0.75]
\draw[double] (0.5,0.5) -- (1.5,0.5);
\draw[ar] (0,0) -- (0.5,0.5);
\draw[ar] (0,1) -- (0.5,0.5);
\draw[ar] (2,0) -- (1.5,0.5);
\draw[ar] (2,1) -- (1.5,0.5);
\end{tikzpicture}\end{split} \hskip10pt \text{and}\hskip10pt \begin{split}\begin{tikzpicture}[ar/.style={postaction={decorate,decoration={markings,mark=at position .5 with {\arrow{>}}}}},scale=0.75]
\draw[double] (0.5,0.5) -- (0.5,1.5);
\draw[ar] (0.5,0.5) -- (0,0);
\draw[ar] (0.5,0.5) -- (1,0);
\draw[ar] (0.5,1.5) -- (0,2);
\draw[ar] (0.5,1.5) -- (1,2);
\end{tikzpicture}\end{split} =\hskip5pt \begin{split}
\begin{tikzpicture}[ar/.style={postaction={decorate,decoration={markings,mark=at position .5 with {\arrow{>}}}}},scale=0.75]
\draw[double] (0.5,0.5) -- (1.5,0.5);
\draw[ar] (0.5,0.5) -- (0,0);
\draw[ar] (0.5,0.5) -- (0,1);
\draw[ar] (1.5,0.5) -- (2,0);
\draw[ar] (1.5,0.5) -- (2,1);
\end{tikzpicture}\end{split} 
\end{equation}

\begin{proof}
The following sequence of diagrams all share the same invariant vector:
$$
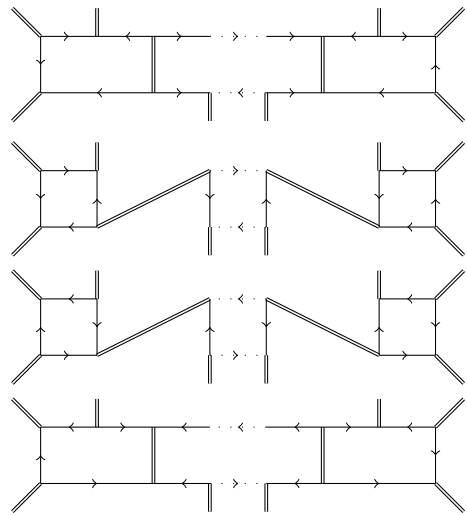
\begin{tikzpicture}[ar/.style={postaction={decorate,decoration={markings,mark=at position .5 with {\arrow{>}}}}},scale=0.75]
\draw[double] (0.5,0.5) -- (1,1) (0.5,2.5) -- (1,2) (2,2.5) -- (2,2) (3,2) -- (3,1) (4,1) -- (4,0.5) (7,2.5) -- (7,2) (6,2) -- (6,1) (5,1) -- (5,0.5) (8,2) -- (8.5,2.5) (8,1) -- (8.5,0.5);
\draw[ar] (1,2) -- (1,1);
\draw[ar] (3,1) -- (1,1);
\draw[ar] (1,2) -- (2,2);
\draw[ar] (3,2) -- (2,2);
\draw[ar] (3,2) -- (4,2);
\draw[ar] (3,1) -- (4,1);
\draw[ar,loosely dotted] (4,2) -- (5,2);
\draw[ar,loosely dotted] (5,1) -- (4,1);
\draw[ar] (5,2) -- (6,2);
\draw[ar] (5,1) -- (6,1);
\draw[ar] (8,1) -- (6,1);
\draw[ar] (7,2) -- (6,2);
\draw[ar] (7,2) -- (8,2);
\draw[ar] (8,1) -- (8,2);
\end{tikzpicture}$$
$$\begin{tikzpicture}[ar/.style={postaction={decorate,decoration={markings,mark=at position .5 with {\arrow{>}}}}},scale=0.75]
\draw[double] (0.5,0.5) -- (1,1) (0.5,2.5) -- (1,2) (2,2.5) -- (2,2) (4,2) -- (2,1) (4,1) -- (4,0.5) (7,2.5) -- (7,2) (5,2) -- (7,1) (5,1) -- (5,0.5) (8,2) -- (8.5,2.5) (8,1) -- (8.5,0.5);
\draw[ar] (1,2) -- (1,1);
\draw[ar] (2,1) -- (1,1);
\draw[ar] (1,2) -- (2,2);
\draw[ar] (2,1) -- (2,2);
\draw[ar] (4,2) -- (4,1);
\draw[ar,loosely dotted] (4,2) -- (5,2);
\draw[ar,loosely dotted] (5,1) -- (4,1);
\draw[ar] (5,1) -- (5,2);
\draw[ar] (7,2) -- (8,2);
\draw[ar] (7,2) -- (7,1);
\draw[ar] (8,1) -- (7,1);
\draw[ar] (8,1) -- (8,2);
\end{tikzpicture}$$
$$\begin{tikzpicture}[ar/.style={postaction={decorate,decoration={markings,mark=at position .5 with {\arrow{>}}}}},scale=0.75]
\draw[double] (0.5,0.5) -- (1,1) (0.5,2.5) -- (1,2) (2,2.5) -- (2,2) (4,2) -- (2,1) (4,1) -- (4,0.5) (7,2.5) -- (7,2) (5,2) -- (7,1) (5,1) -- (5,0.5) (8,2) -- (8.5,2.5) (8,1) -- (8.5,0.5);
\draw[ar] (1,1) -- (1,2);
\draw[ar] (1,1) -- (2,1);
\draw[ar] (2,2) -- (1,2);
\draw[ar] (2,2) -- (2,1);
\draw[ar] (4,1) -- (4,2);
\draw[ar,loosely dotted] (5,2) -- (4,2);
\draw[ar,loosely dotted] (4,1) -- (5,1);
\draw[ar] (5,2) -- (5,1);
\draw[ar] (8,2) -- (7,2);
\draw[ar] (7,1) -- (7,2);
\draw[ar] (7,1) -- (8,1);
\draw[ar] (8,2) -- (8,1);
\end{tikzpicture}$$
$$\begin{tikzpicture}[ar/.style={postaction={decorate,decoration={markings,mark=at position .5 with {\arrow{>}}}}},scale=0.75]
\draw[double] (0.5,0.5) -- (1,1) (0.5,2.5) -- (1,2) (2,2.5) -- (2,2) (3,2) -- (3,1) (4,1) -- (4,0.5) (7,2.5) -- (7,2) (6,2) -- (6,1) (5,1) -- (5,0.5) (8,2) -- (8.5,2.5) (8,1) -- (8.5,0.5);
\draw[ar] (1,1) -- (1,2);
\draw[ar] (1,1) -- (3,1);
\draw[ar] (2,2) -- (1,2);
\draw[ar] (2,2) -- (3,2);
\draw[ar] (4,2) -- (3,2);
\draw[ar] (4,1) -- (3,1);
\draw[ar,loosely dotted] (5,2) -- (4,2);
\draw[ar,loosely dotted] (4,1) -- (5,1);
\draw[ar] (6,2) -- (5,2);
\draw[ar] (6,1) -- (5,1);
\draw[ar] (6,1) -- (8,1);
\draw[ar] (6,2) -- (7,2);
\draw[ar] (8,2) -- (7,2);
\draw[ar] (8,2) -- (8,1);
\end{tikzpicture}$$

The first and last are the two choices we could make in the construction of the triangular diagrams. Since the internal edges of weight $\omega_2$ have all other edges adjacent incoming or outgoing, we can apply equation~\ref{sl4ker2}. This shows that the second and third diagrams result in the same invariant vector as the first and last respectively. The resulting two diagrams are a sequence of squares linking by edges of weight $\omega_2$. We can apply equation~\ref{sl4ker1} to move between these two diagrams. Thus the second and third diagrams have the same invariant vector.
\end{proof}

Since the resulting set of webs is a basis, the hope would be to use this basis to prove the conjecture of Kim \cite{Kim:thesis} on the generators of the kernel of the map $\Psi$. If it can be shown that any web can be reduced to a sum of the given basis webs by Kim's relations, then the set of kernel relations would be complete. Unfortunately there does not seem to be any global criterion which picks out the basis webs from the set of coherent webs. For instance the webs in Figures~\ref{fig:nonminimal} and \ref{fig:minimal} both have associated sequence $(0,\omega_1,\omega_3+\omega_1,\omega_2+\omega_3,\omega_2,0)$. Figure~\ref{fig:nonminimal} is build via triangular diagrams is neither minimal with respect to the number of vertices or faces.

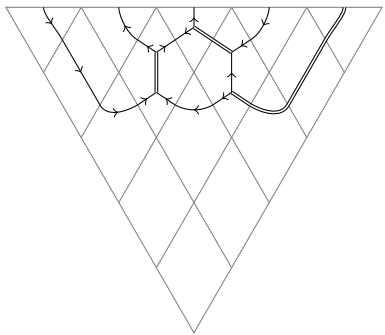
\begin{figure*}\begin{tikzpicture}[ar/.style={postaction={decorate,decoration={markings,mark=at position .5 with {\arrow{>}}}}}]
\draw[gray] (0,0) -- (60:5) -- (120:5) -- cycle;
\draw[gray] (60:1) -- +(120:4);
\draw[gray] (60:2) -- +(120:3);
\draw[gray] (60:3) -- +(120:2);
\draw[gray] (60:4) -- +(120:1);
\draw[gray] (120:1) -- +(60:4);
\draw[gray] (120:2) -- +(60:3);
\draw[gray] (120:3) -- +(60:2);
\draw[gray] (120:4) -- +(60:1);

\draw[ar] ($ (120:5) + (0.5,0) $) .. controls ($ (120:5) + (0.5,0) + (0,-0.1) $) and ($ (120:4.2) + (60:0.5) $) .. ($ (120:4) + (60:0.5) $);
\draw[ar] ($ (120:4) + (60:0.5) $) -- +(300:1);
\draw[ar] ($ (120:3) + (60:0.5) $) .. controls ($ (120:2.8) + (60:0.5) $) and ($ (120:2.6) + (60:0.8) $) .. ($ (120:2.5) + (60:1) $);

\draw[ar] ($ (120:3) + (60:1.5) $) .. controls ($ (120:3.2) + (60:1.4) $) and ($ (120:4) + (60:1) + (0.5,-0.1) $) .. ($ (120:4) + (60:1) + (0.5,0) $);
\draw[ar] ($ (120:2.5) + (60:1) $) -- ($ (120:2) + (60:1) + (0,0.6) $);
\draw[ar] ($ (120:2) + (60:1.5) $) -- ($ (120:2) + (60:1) + (0,0.6) $);
\draw[ar] ($ (120:3) + (60:2) - (0,0.6) $) -- ($ (120:3) + (60:1.5) $);
\draw[ar] ($ (120:3) + (60:2) - (0,0.6) $) -- ($ (120:2.5) + (60:2) $);
\draw[double] ($ (120:2) + (60:1) + (0,0.6) $) -- ($ (120:3) + (60:2) - (0,0.6) $);
\draw[ar]  ($ (120:1.5) + (60:2) $) .. controls ($ (120:1.6) + (60:1.8) $) and ($ (120:1.8) + (60:1.6) $) .. ($ (120:2) + (60:1.5) $);

\draw[ar] ($ (120:2) + (60:2) + (0,0.6) $) -- ($ (120:2.5) + (60:2) $);
\draw[ar] ($ (120:2) + (60:2) + (0,0.6) $) -- ($ (120:3) + (60:2) + (0.5,0) $);
\draw[double] ($ (120:2) + (60:2) + (0,0.6) $) -- ($ (120:2) + (60:2.5) $);

\draw[double] ($ (60:2.5) + (120:1) $) -- ($ (60:2) + (120:1) + (0,0.6) $);
\draw[ar] ($ (60:2) + (120:1) + (0,0.6) $) -- ($ (60:2) + (120:1.5) $);
\draw[ar] ($ (60:3) + (120:1.5) $) -- ($ (60:3) + (120:2) - (0,0.6) $);
\draw[double] ($ (60:3) + (120:2) - (0,0.6) $) -- ($ (60:2.5) + (120:2) $);
\draw[ar] ($ (60:2) + (120:1) + (0,0.6) $) -- ($ (60:3) + (120:2) - (0,0.6) $);
\draw[double] ($ (60:3) + (120:0.5) $) .. controls ($ (60:2.8) + (120:0.5) $) and ($ (60:2.6) + (120:0.8) $) .. ($ (60:2.5) + (120:1) $);

\draw[ar] ($ (60:4) + (120:1) + (-0.5,0) $) .. controls ($ (60:4) + (120:1) + (-0.5,-0.1) $) and ($ (60:3.2) + (120:1.4) $) .. ($ (60:3) + (120:1.5) $) ;
\draw[double] ($ (60:4) + (120:0.5) $) -- +(240:1);

\draw[double] ($ (60:5) + (-0.5,0) $) .. controls ($ (60:5) + (-0.5,0) + (0,-0.1) $) and ($ (60:4.2) + (120:0.5) $) .. ($ (60:4) + (120:0.5) $);
\end{tikzpicture}
\caption{The triangular diagram associated to the minuscule Littelmann path $(0,\omega_1,\omega_3+\omega_1,\omega_2+\omega_3,\omega_2,0)$.}\label{fig:nonminimal} \end{figure*}

\begin{figure*} \begin{tikzpicture}[ar/.style={postaction={decorate,decoration={markings,mark=at position .5 with {\arrow{>}}}}}]
\draw[gray] (-0.5,0) -- (4.5,0);
\draw[ar] (0,0) .. controls (0,-2) and (1.8,-2.9) .. (2,-3);
\draw[double] (4,0) .. controls (4,-2) and (2.2,-2.9) .. (2,-3);
\draw[ar] (2,0) .. controls (2,-0.6) and (2.35,-0.95) .. (2.5,-1);
\draw[ar] (3,0) .. controls (3,-0.6) and (2.65,-0.95) .. (2.5,-1);
\draw[ar] (2,-2) .. controls (1.8,-1.9) and (1,-1) .. (1,0);
\draw[double] (2.5,-1) .. controls (2.5,-1.5) and (2.2,-1.9) .. (2,-2);
\draw[ar] (2,-2) -- (2,-3);
\end{tikzpicture}
\caption{A diagram associated to the minuscule Littelmann path $(0,\omega_1,\omega_3+\omega_1,\omega_2+\omega_3,\omega_2,0)$ with a minimal number of faces.}\label{fig:minimal} \end{figure*}
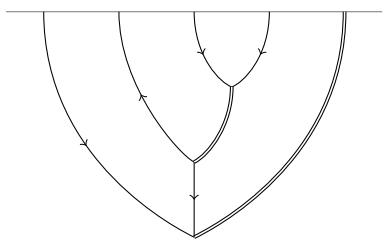
\bibliographystyle{hamsplain}
\bibliography{basiswebs}

\end{document}